\documentclass{amsart}

\usepackage[english]{babel}
\usepackage{csquotes} %Zitate
\usepackage{amssymb,amsmath,amsthm}
\usepackage{dsfont}
\usepackage{graphicx}
\usepackage[numbers]{natbib}
\usepackage{color}
\usepackage{pdflscape}
\usepackage{morefloats}

\newcommand{\ignore}[1]{}

\newcommand{\N}{{\mathds N}}
\newcommand{\R}{{\mathds R}}
\newcommand{\Z}{{\mathds Z}}
\newcommand{\Var}{{\rm Var}}
\newcommand{\Cov}{{\rm Cov}}
\newcommand{\corr}{{\rm corr}}

\DeclareMathOperator{\E}{E}

\DeclareMathOperator{\e}{e}
\DeclareMathOperator{\rank}{rank}

\newtheorem{Satz}{Theorem}
\newtheorem*{Satz*}{Theorem A}
\newtheorem{Assumption}{Assumption}
\newtheorem{Lem}{Lemma}

\theoremstyle{definition}

\newtheorem{Exa}{Example}

\begin{document}

\title[Subsampling under Long Range Dependence]{Subsampling for General Statistics under Long Range Dependence with application to change point analysis}

\author{Annika Betken}
\thanks{
supported by the German National Academic Foundation, Collaborative Research Center SFB 823 {\em Statistical modelling of nonlinear dynamic processes}.}
\address{Ruhr-Universit\"at Bochum, Germany}
\email{annika.betken@rub.de}

\author{Martin Wendler}
\address{Ernst-Moritz-Arndt-Universit\"at Greifswald, Germany}
\email{martin.wendler@uni-greifswald.de}

\begin{abstract} In the statistical inference for long range dependent time series the shape of the limit distribution typically depends on unknown parameters. Therefore, we propose to use subsampling. We show the validity of subsampling for general statistics and long range dependent subordinated Gaussian processes which satisfy mild regularity conditions. We apply our method to a self-normalized change-point test statistic so that we can test for structural breaks in long range dependent time series without having to estimate any nuisance parameter. The finite sample properties are investigated in a simulation study. We analyze three data sets and compare our results to the conclusions of other authors.
\end{abstract}



\maketitle

\section{Introduction}\label{sec:intro}

\subsection{Long Range Dependence}

While most statistical research is done for independent data or short memory time series, in many applications there are also time series with long memory in the sense of slowly decaying correlations: in hydrology (starting with the work of Hurst \citep{Hurst1956}), in finance (e.g. Lo \cite{Lo1989}), in the analysis of network traffic (e.g. Leland, Taqqu, Willinger and Wilson \citep{Leland1994}) and in many other fields of research.

As model of dependent time series we will consider subordinated Gaussian processes: Let $(\xi_n)_{n\in\N}$ be a stationary sequence of centered Gaussian variables with $\Var(\xi_n)=1$ and covariance function $\gamma$ satisfying
\begin{align}\label{autocovariances}
\gamma(k):=\Cov(\xi_1,\xi_{k+1})=k^{-D}L_\gamma(k)
\end{align}    
for $D>0$ and a slowly varying function $L_{\gamma}$. If $D<1$,  the spectral density $f$ of $(\xi_n)_{n\in\N}$ is not continuous, but has a pole at $0$. The spectral density has the form
\begin{equation*}
f(x)=|x|^{D-1}L_f(x)
\end{equation*} 
for a function $L_f$ which is slowly varying at the origin (see Proposition 1.1.14 in Pipiras and Taqqu \citep{TaqquPipiras2011}).

Furthermore, let $G:\R\rightarrow\R$ be a measurable function such that $\E[G^2(\xi_1)]<\infty$. The stochastic process $(X_n)_{n\in\N}$ given by
\begin{equation*}
X_n:=G(\xi_n)
\end{equation*}
is called long range dependent if $\sum_{n=0}^\infty|\Cov(X_1,X_{n+1})|=\infty$, and short range dependent if $\sum_{n=0}^\infty|\Cov(X_1,X_{n+1})|<\infty$.

In limit theorems for the partial sum $S_n=\sum_{i=1}^nX_i$, the normalization  and the shape of the limit distribution not only depend  on the decay of the covariances $\gamma(k)$ as $k\rightarrow\infty$, but also on the function $G$. More precisely, Taqqu \citep{Taqqu1979} and Dobrushin and Major \citep{DobrushinMajor1979} independently proved that
\begin{equation*}
\frac{1}{L_\gamma(n)^{r/2}n^{H}}\sum_{i=1}^n\left(X_i-\E[X_i]\right)\Rightarrow C(r,H)g_rZ_{r,H}(1)
\end{equation*}
if the Hurst parameter $H:=\max\{1-\frac{rD}{2},\frac{1}{2}\}$ is greater than $\frac{1}{2}$. Here, $r$ denotes the Hermite rank of the function $G$, $C(r, H)$  is a constant, $g_r$ is the first non-zero coefficient in the expansion of $G$ as a sum of Hermite polynomials and $Z_{r,H}$ is a Hermite process. For more details on Hermite polynomials and limit theorems for subordinated Gaussian processes we recommend the book of Pipiras and Taqqu \citep{TaqquPipiras2011}. In this case ($rD<1$), the process $(X_n)_{n\in\N}$ is long range dependent as the covariances are not summable. Note that the limiting random variable $C(r,H)Z_{r,H}(1)$ is Gaussian only if the Hermite rank $r=1$.

If $rD=1$, the process $(X_n)_{n\in\N}$ might be short or long range dependent according to the slowly varying function $L_{\gamma}$. If $rD>1$, the process is short range dependent. In this case, the partial sum $\sum_{i=1}^n(X_i-\E[X_i])$ has (with  proper normalization) always a Gaussian limit. 

There are other models for long memory processes: Fractionally integrated autoregressive moving average processes can show long range dependence, see Granger and Joyeux \citep{Granger1980}. General linear processes with slowly decaying coefficients were studied by Surgailis \citep{Surgailis1982}.

\subsection{Subsampling}

For practical applications the parameters $D$, $r$ and the slowly varying function $L_\gamma$ are unknown and thus the scaling needed in the limit theorems and the shape of the asymptotic distribution are not known, either. That makes it difficult to use the asymptotic distribution for statistical inference. The situation gets even more complicated if one is not interested in partial sums, but in nonlinear statistical functionals. For example, $U$-statistics can have a limit distribution which is a linear combination of random variables related to different Hermite ranks, see Beutner and Z\"ahle \citep{BeutnerZaehle2014}. Self-normalized statistics typically converge to quotients of two random variables (e.g. McElroy and Politis \citep{McelroyPolitis2007}). The change-point test proposed by Berkes, Horv\'{a}th, Kokoszka and Shao \citep{Berkes2006} converges  to the supremum of a fractional Brownian bridge under the alternative hypothesis.

To overcome the problem of the unknown shape of the limit distribution and to avoid the estimation of nuisance parameters, one would like to use nonparametric methods. However, Lahiri \citep{Lahiri1993} has shown that the popular moving block bootstrap might fail under long range dependence. Another nonparametric approach is subsampling (also called sampling window method), first studied by Politis and Romano \citep{Politisromano1994l}, Hall and Jing \citep{HallJing1996}, and Sherman and Carlstein \citep{ShermanCarlstein1996}. The idea is the following: Let $T_n=T_n(X_1,\ldots,X_n)$ be a series of statistics converging in distribution to a random variable $T$. However,  as we typically just have one sample, we observe only one realization of $T_n$ and therefore cannot estimate the distribution of $T_n$. If $l=l_n$ is a sequence with $l_n\rightarrow\infty$ and $l_n=o(n)$, then $T_{l}$ also converges in distribution to $T$ and we have multiple (though dependent) realizations  $T_l(X_1,\ldots,X_l)$, $T_l(X_2,\ldots,X_{l+1})$,$\ldots$, $T_l(X_{n-l+1},\ldots,X_n)$, which  can be used to calculate the empirical distribution function.

Note that we do not need to know the limit distribution. In our example (self-normalized change point test statistic, see Section 3), the shape of the distribution depends on two unknown parameters, but we can still apply subsampling. However, for other statistics, one needs an unknown scaling to achieve convergence. If this is the case, one has to estimate the scaling parameters before applying subsampling.

Under long range dependence the validity of subsampling for the sample mean $\bar{X}=\frac{1}{n}\sum_{i=1}^nX_i$ has been investigated in the literature starting with Hall, Jing and Lahiri \citep{Halletal1998} for subordinated Gaussian processes. Nordman and Lahiri \citep{NordmanLahiri2005} and Zhang, Ho, Wendler and Wu \citep{Zhangetal2013} studied linear processes with slowly decaying coefficients. For the case of Gaussian processes an alternative proof  can be found in the book of Beran, Feng, Ghosh and  Kulik \citep{BeranFengGhoshKulik2013}.

It was noted by Fan \citep{Fan2012} that the proof in \citep{Halletal1998} can be easily generalized to other statistics than the sample mean. However, the assumptions on the Gaussian process are restrictive (see also  \citep{McelroyPolitis2007}). Their conditions imply that the sequence $(\xi_n)_{n\in\N}$ is completely regular, which might hold for some special cases (see Ibragimov and Rozanov \citep{IbragimovRozanov1978}), but excludes many examples:

\begin{Exa}[Fractional Gaussian Noise]\label{exa_fGn} Let $(B_H(t))_{t\in[0,\infty)}$ be a fractional Brownian motion, i.e. a centered, self-similar Gaussian process with covariance function
\begin{equation*}
\E\left[B_H(t)B_H(s)\right]=\frac{1}{2}\left(|t|^{2H}+|s|^{2H}-|t-s|^{2H}\right)
\end{equation*}
for some $H\in(\frac{1}{2},1)$. Then, $(\xi_n)_{n\in\N}$ given by $\xi_n=B_H(n)-B_H(n-1)$ is called fractional Gaussian noise. By self-similarity we have
\begin{align*}
\corr\bigg(\sum_{i=1}^n\xi_{i},\sum_{j=2n+1}^{3n}\xi_{j}\bigg)
&=\corr\left(B_H(n),B_H(3n)-B_H(2n)\right)\\
&=\corr\left(B_H(1),B_H(3)-B_H(2)\right).
\end{align*}
As a result,  the correlations of linear combinations of observations in the past and future do not vanish if the gap between past and future grows. Thus, fractional Gaussian noise is not completely regular.
\end{Exa}

Jach, McElroy and Politis \citep{Jachetal2012} provided a more general result on the validity of subsampling. They assume that the function $G$ has Hermite rank 1, that $G$ is invertible and Lipschitz-continuous and that the process $(\xi_n)_{n\in\N}$ has a causal representation as a functional of an independent sequence of random variables. These assumptions are difficult to check in practice. Moreover, although not explicitly stated in \citep{Jachetal2012}, the statistic $T_n$ has to be Lipschitz-continuous (uniformly in $n$),  which is not satisfied by many robust estimators (see Section \ref{sec:appl} for an example).

The main aim of this paper is to establish the validity of the subsampling method for general statistics $T_n$ without any assumptions on the continuity of the statistic, on the function $G$ and only mild assumptions on the Gaussian process $(\xi_n)_{n\in\N}$. Independently of our research, similar theorems have been proved by Bai, Taqqu and Zhang \citep{bai2016unified}. We will discuss their results after our main theorem in Section \ref{sec:main}. In Section \ref{sec:appl} we will apply our theorem to a self-normalized, robust change-point statistic. The finite sample properties of this test will be investigated  in a simulation study in Section \ref{sec:simu}. Finally, the proof of the main result and the lemmas needed can be found in Section \ref{sec:proof}.

\section{Main Results}\label{sec:main}

\subsection{Statement of the Theorem}

For a statistic $T_n=T_n(X_1,\ldots,X_n)$ the subsampling estimator $\hat{F}_{l,n}$ of the distribution function $F_{T_n}$ with $F_{T_n}(t)=P(T_n\leq t)$ is defined in the following way: For $t\in\R$
let
\begin{equation*}
\hat{F}_{l,n}(t)=\frac{1}{n-l+1}\sum\limits_{i=1}^{n-l+1}1_{\left\{T_l(X_i,\ldots,X_{i+l-1})\leq t\right\}}.
\end{equation*}
Our first assumption guarantees the convergence of the distribution function  $F_{T_n}$:
\begin{Assumption}\label{convergence of T_n}
 $(X_n)_{n\in\N}$ is a stochastic process and $(T_n)_{n\in\N}$ is a sequence of statistics such that $T_n\Rightarrow T$ in distribution as $n\rightarrow\infty$ for a random variable $T$ with distribution function $F_T$.
\end{Assumption}
This is a standard assumption for subsampling, see for example  \citep{Politisromano1994l}. If the distribution does not converge, we cannot expect the distribution of $T_l$ to be close to the distribution of $T_n$.

Next, we will formulate our conditions on the sequence  of random variables $(X_n)_{n\in\N}$:
\begin{Assumption}\label{dgp}
 $X_n=G(\xi_n)$ for a measurable function $G$ and a stationary, Gaussian process $(\xi_n)_{n\in\N}$ with covariance function
\begin{equation*}
\gamma(k):=\Cov(\xi_1,\xi_{1+k})=k^{-D}L_{\gamma}(k)
\end{equation*}
such that the following conditions hold:
\begin{enumerate}
\item $D\in(0,1]$ and $L_\gamma$ is a slowly varying function with
\begin{align*} 
\max\limits_{\tilde{k}\in\{k+1,\ldots,k+2l'-1\}}\left|L_{\gamma}(k)-L_{\gamma}(\tilde{k})\right|\leq K\frac{l'}{k}\min\left\{L_\gamma(k),1\right\}
\end{align*}
for a constant $K<\infty$ and all $l'\in\{l_k,\ldots,k\}$.
\item $(\xi_n)_{n\in\N}$ has a spectral density $f$ with $f(x)=|x|^{D-1}L_f(x)$ for a slowly varying function $L_f$ which is bounded away from $0$ on  $\left[0, \pi\right]$ such that $\lim_{x\rightarrow0}L_f(x)\in(0,\infty]$ exists.
\end{enumerate}
\end{Assumption}
While we have some regularity conditions on the underlying Gaussian process $(\xi_n)_{n\in\N}$, we do not impose any conditions on the function $G$: no finite moments or continuity are required, so that our results are applicable for heavy-tailed random variables and  robust test statistics. In the next subsection we will show that Assumption \ref{dgp} holds for some standard examples of long range dependent Gaussian processes.

Furthermore, we need a restriction on the growth rate of the block length $l$:
\begin{Assumption}\label{blocklength} Let $(l_n)_{n\in\N}$ be a non-decreasing sequence of integers such that $l=l_n\rightarrow\infty$ as $n\rightarrow\infty$ and $l_n=\mathcal{O}\big(n^{(1+D)/2-\epsilon}\big)$ for some $\epsilon>0$.
\end{Assumption}
If the dependence of the underlying process $(\xi_n)_{n\in\N}$ gets stronger, the range of possible values for $l$ gets smaller. A popular choice for the block length is $l\approx C\sqrt{n}$ (see for example  \citep{Halletal1998}), which is allowed for all $D\in(0,1]$. Now, we can state our main result:
\begin{Satz}\label{main result}
Under  Assumptions \ref{convergence of T_n}, \ref{dgp} and \ref{blocklength}  we have
\begin{equation*}
F_{T_n}(t)-\hat{F}_{l,n}(t)\overset{\mathcal{P}}{\longrightarrow}0
\end{equation*}
as $n\rightarrow\infty$ for all points of continuity $t$ of $F_T$.
If $F_T$ is continuous, then
\begin{equation*}
\sup\limits_{t\in \mathbb{R}}\left|F_{T_n}(t)-\hat{F}_{l,n}(t)\right|\overset{\mathcal{P}}{\longrightarrow}0.
\end{equation*}
\end{Satz}
As a result,  we have a consistent estimator for the distribution function of $T_n$. It is possible to build tests and confidence intervals based on this estimator.

If $D>1$, the process $(\xi_n)_{n\in\N}$ is strongly mixing due to Theorem 9.8 in the book of Bradley \citep{bradley2007introduction}. The statements of Theorem \ref{main result} hold by Corollary 3.2 in \citep{Politisromano1994l} for any block length $l$ satisfying $l\rightarrow\infty$ and $l=o(n)$.

In a recent article, Bai et al. \citep{bai2016unified} have shown that subsampling is consistent for long range dependent Gaussian processes without any extra assumptions on the slowly varying function $L_f$, but with a stronger restriction on the block size $l$, namely $l=o(n^{2-2H}L_\gamma(n))$. In another article by Bai and Taqqu \citep{bai2015canonical}, the validity of subsampling is shown under the mildest possible assumption on the block length ($l=o(n)$). The condition on the spectral density is slightly stronger than our condition, the case $\lim_{x\rightarrow 0}L_f(x)=\infty$ is not allowed.

\subsection{Examples for our Assumptions}

We will now give two examples of Gaussian processes satisfying Assumption \ref{dgp}:

\begin{Exa}[Fractional Gaussian Noise]\label{exa_fGn2} Fractional Gaussian Noise $(\xi_n)_{n\in\N}$ with Hurst parameter $H$ as introduced in Example \ref{exa_fGn} has the covariance function
\begin{equation*}
\gamma(k)=\frac{1}{2}\left(|k-1|^{2H}-2|k|^{2H}+|k+1|^{2H}\right)= H(2H-1)\left(k^{-D}+h(k)k^{-D-1}\right)
\end{equation*}
for $D=2-2H$ and a function $h$ bounded by a constant $M<\infty$. This can be easily seen by means of a Taylor expansion. Hence,  $L_{\gamma}(k)=H(2H-1)(1+h(k)/k)$ and for all $\tilde{k}\geq k$
\begin{equation*}
\left|L_\gamma(k)-L_\gamma(\tilde{k})\right|\leq H(2H-1)\bigg|\frac{h(k)}{k}-\frac{h(\tilde{k})}{\tilde{k}}\bigg|\leq H(2H-1)\frac{M}{k}=:K\frac{1}{k}.
\end{equation*}
This implies part 1 of Assumption \ref{dgp}.  For the second part note that the spectral density $f$ corresponding to  fractional Gaussian noise is given by
\begin{align*}
f(\lambda)
&=C(H)(1-\cos(\lambda))\sum_{k=-\infty}^\infty\left|\lambda+2k\pi\right|^{D-3}\\
&=\lambda^{D-1}C(H)\frac{1-\cos(\lambda)}{\lambda^2}\frac{\sum_{k=-\infty}^\infty\left|\lambda+2k\pi\right|^{D-3}}{\lambda^{D-3}},
\end{align*}
see Sinai \cite{sinai1976self}. The slowly varying function
\begin{equation*}
L_f(\lambda)=C(H)\frac{1-\cos(\lambda)}{\lambda^2}\frac{\sum_{k=-\infty}^\infty\left|\lambda+2k\pi\right|^{D-3}}{\lambda^{D-3}}
\end{equation*}
is bounded away from 0 because this  holds for the first factor $(1-\cos(\lambda))/\lambda^2$ and since
\begin{equation*}
\frac{\sum_{k=-\infty}^\infty\left|\lambda+2k\pi\right|^{D-3}}{\lambda^{D-3}}\geq \frac{\left|\lambda+0\pi\right|^{D-3}}{\lambda^{D-3}}=1.
\end{equation*}
\end{Exa}

\begin{Exa}[Gaussian FARIMA processes]\label{exa_arfima} Let $(\varepsilon_n)_{n\in\Z}$ be Gaussian white noise with variance $\sigma^2=\Var(\varepsilon_0)$. Then, for  $d\in(0,1/2)$, a FARIMA($0$, $d$, $0$) process ($\xi_n)_{n\in\N}$ is given by
\begin{equation*}
\xi_n=\sum_{j=0}^\infty \frac{\Gamma(j+d)}{\Gamma(j+1)\Gamma(d)}\varepsilon_{n-j}.
\end{equation*}
According to Pipiras and Taqqu \cite{TaqquPipiras2011}, Section 1.3, it has the specral density
\begin{equation*}
f(\lambda)=\frac{\sigma^2}{2\pi}|1-e^{-i\lambda}|^{-2d}=|\lambda|^{D-1}\frac{\sigma^2}{2\pi}\left(\frac{|\lambda|}{|1-e^{-i\lambda}|}\right)^{1-D}
\end{equation*}
with $D=1-2d\in(0,1)$. As $|1-e^{-i\lambda}|\leq \lambda$, part 2 of Assumption \ref{dgp} holds. For part 1 we have by Corollary 1.3.4 of \cite{TaqquPipiras2011} that
\begin{equation*}
\gamma(k)=\sigma^2\frac{\Gamma(1-2d)}{\Gamma(1-d)\Gamma(d)}\frac{\Gamma(k+d)}{\Gamma(k-d+1)}.
\end{equation*}
Recall that by the Stirling formula $\Gamma(x)=\big(\frac{2\pi}{x}\big)^{1/2}\big(\frac{x}{e}\big)^x\big(1+\mathcal{O}(x^{-1})\big)$. Consequently,
\begin{equation*}
\gamma(k)=\sigma^2\frac{\Gamma(1-2d)}{\Gamma(1-d)\Gamma(d)}e^{-2d+1}k^{2d-1}\Big(\frac{k+d}{k}\Big)^{k+d}\Big(\frac{k}{k-d+1}\Big)^{k-d+1}\Big(1+\mathcal{O}\big(\frac{1}{k}\big)\Big).
\end{equation*}
Using a Taylor expansion of $(k+d)\big(\log(k+d)-\log(k)\big)+(k-d+1)\big(\log(k)-\log(k-d+1)\big)$, it easily follows that
\begin{equation*}
\gamma(k)=k^{-D}L_\gamma(k)
\end{equation*}
with $L_\gamma(k)=C+\mathcal{O}(1/k)$ for some constant $C$. Part 1 of Assumption \ref{dgp} follows in the same way as in Example \ref{exa_fGn2}.
\end{Exa}

It would be interesting to know, if the sampling window method is also consistent  for long range dependent linear processes and general statistics without the assumption of Gaussianity. However,  this seems to be a very difficult problem and is beyond the scope of this article.

\section{Applications}\label{sec:appl}

\subsection{Robust, Self-Normalized Change-Point Test}

In this paper, the main motivation for considering subsampling procedures in order to approximate the distribution of test statistics consists in avoiding the choice of unknown parameters. As an example we will consider  a self-normalized test statistic that can be applied  to detect changes in the mean of long range dependent and heavy-tailed time series.

Given observations $X_1, \ldots, X_n$ with $X_i=\mu_i +G(\xi_i)$ we are concerned with a decision on the change-point problem
\begin{align*}
&\operatorname{\mathbf{H}}: \mu_1=\ldots =\mu_n\\
\intertext{against}
&\operatorname{\mathbf{A}}:  \mu_1=\ldots =\mu_k\neq \mu_{k+1}=\ldots =\mu_n\ \ \text{for some } k\in \left\{1, \ldots, n-1\right\}.
\end{align*}
Under the hypothesis $\operatorname{\mathbf{H}}$ we assume that the data generating process $\left(X_n\right)_{n \in \mathbb{N}}$ is stationary, while under the alternative $\operatorname{\mathbf{A}}$ there is a change in location at an unknown point in time. This problem has been widely studied: Cs\"{o}rg\H{o} and Horv\'{a}th \citep{CsorgoHorvath1997} give an overview of parametric and  non-parametric methods that can be applied in order to detect change-points in independent data.

Many commonly used testing procedures are based on Cusum (cumulative sum) test statistics, but when applied to data sets generated by long range dependent processes, these change-point tests often falsely reject the hypothesis of no change in the mean (see also Baek and Pipiras \citep{baek2014distinguishing}). Furthermore, the performance of Cusum-like change-point tests is sensitive to outliers in the data. 

In contrast, testing procedures that are based on rank statistics have the advantage of not being sensitive to outliers in the data. Rank-based  tests were introduced by  Antoch, Hu\v{s}kov\'{a}, Janic and Ledwina \citep{Antochetal2008} for detecting changes in the distribution function of independent random variables. Wilcoxon-type rank tests have been studied by Wang \citep{Wang2008} in the presence of linear long memory time series and by Dehling, Rooch and Taqqu \cite{DehlingRoochTaqqu2013a} for subordinated Gaussian sequences. 

Note that the normalization of the Wilcoxon change-point test statistic as proposed in \cite{DehlingRoochTaqqu2013a} depends on the slowly varying  function $L_{\gamma}$, the LRD parameter $D$ and the Hermite rank $r$ of the class of functions $1_{\{X_i\leq x\}}-F(x)$, $x\in \mathbb{R}$. Although many authors assume $r=1$ and while there are well-tried methods to estimate $D$, estimating $L_{\gamma}$ does not seem to be an easy task. For this reason, the Wilcoxon change-point test does not seem to be suitable for applications to real data.

To  avoid these issues, Betken \citep{Betken2014} proposes an alternative normalization for the Wilcoxon change-point test. This normalization approach has originally been established by Lobato \citep{Lobato2001} for decision on the hypothesis  that a short range dependent stochastic process is uncorrelated up to a lag of a certain order.
In change-point analysis, the normalization has recently been applied to several test statistics: Shao and Zhang \citep{ShaoZhang2010} define a self-normalized Kolmogorov-Smirnov test statistic that serves to identify changes in the mean of short range dependent time series. Shao \citep{Shao2011} adopted the normalization so as to define an alternative normalization for a Cusum test which detects changes in the mean of short range dependent as well as  long range dependent time series. 

For the definition of the self-normalized Wilcoxon test statistic, we introduce the ranks $R_i:=\rank(X_i)=\sum_{j=1}^n1_{\{X_j\leq X_i\}}$ for $i=1,\ldots,n$. It seems natural to transfer the normalization that has been used in \citep{Shao2011} to the Cusum test statistic of the ranks  in order to establish a self-normalized version of the Wilcoxon test statistic, which is robust to outliers in the data. Therefore, the corresponding two-sample test statistic is defined by
\begin{equation*}
G_n(k)
:=\frac{\sum_{i=1}^kR_i-\frac{k}{n}\sum_{i=1}^nR_i}{\bigg\{\frac{1}{n}\sum_{t=1}^k S_t^2(1,k)+\frac{1}{n}\sum_{t=k+1}^n S_t^2(k+1,n)\bigg\}^{1/2}}, 
\end{equation*}
where 
\begin{equation*}
S_{t}(j, k):=\sum\limits_{h=j}^t\left(R_h-\bar{R}_{j, k}\right)\ \ \text{with }\bar{R}_{j, k}:=\frac{1}{k-j+1}\sum\limits_{t=j}^kR_t.
\end{equation*}
The self-normalized Wilcoxon change-point test rejects the hypothesis for large values of $\max_{k\in \left\{\lfloor n\tau_1\rfloor, \ldots,  \lfloor n\tau_2\rfloor\right\}}\left|G_n(k)\right|$, where $0< \tau_1 <\tau_2 <1$. The proportion of the data that is included in the calculation of the supremum  is restricted by $\tau_1$ and $\tau_2$. A common  choice for these parameters is $\tau_1= 1-\tau_2=0.15$; see Andrews \citep{Andrews1993}.

For long range dependent subordinated Gaussian processes  $\left(X_n\right)_{n\in \mathbb{N}}$, the asymptotic distribution of the test statistic under the hypothesis $\operatorname{\mathbf{H}}$ can be derived by the continuous mapping theorem (see Theorem 1 in \citep{Betken2014}):
\begin{multline*}
T_n(\tau_1, \tau_2):=\max_{k\in \left\{\lfloor n\tau_1\rfloor, \ldots,  \lfloor n\tau_2\rfloor\right\}}\left|G_n(k)\right|\\
 \Rightarrow\sup_{\tau_1\leq\lambda\leq \tau_2}\frac{\left|Z_{r}(\lambda)-\lambda Z_{r}(1)\right|}{\big\{\int_0^\lambda (Z_{r}(t)-\frac{t}{\lambda}Z_{r}(\lambda))^2dt+\int_0^{1-\lambda} (Z_{r}^\star(t)-\frac{t}{1-\lambda}Z_{r}^\star(1-\lambda))^2dt\big\}^{1/2}}.
\end{multline*}
Here, $Z_r$ is an $r$-th order Hermite process with Hurst parameter $H:=\max\{1-\frac{rD}{2},\frac{1}{2}\}$ and $Z_{t}^\star(r)=Z_{r}(1)-Z_{r}(1-t)$. A comparison of  $T_n(\tau_1, \tau_2)$ with the critical values of its limit distribution still presupposes determination of these parameters.   We can bypass the estimation of $D$ and $r$ by applying the subsampling procedure since,  due to the convergence of $T_n(\tau_1, \tau_2)$,  Assumption \ref{convergence of T_n} holds.

Note that even under the alternative $\operatorname{\mathbf{A}}$ (change in location), we have to find the quantiles of the distribution under the hypothesis (stationarity). As the block length $l$ is much shorter than  the sample size $n$,  most blocks will not be contaminated by the change-point so that the distribution of the test statistic will not  change that much.\ignore{ Even for  blocks which cover the change, the power of the test is low (because there are only $l$ observations) and thus the distribution will not change too much.} The accuracy and the power of the test will be investigated by a simulation study in Section \ref{sec:simu}.

If the distribution of $X_i$ is not continuous, there might be ties in the data and consideration of the ranks $R_i=\sum_{j=1}^n1_{\{X_j\leq X_i\}}$ may not be appropriate. We propose to use a modified statistic  based on the modified ranks $\tilde{R}_i=\sum_{j=1}^n(1_{\{X_j< X_i\}}+\frac{1}{2}1_{\{X_j= X_i\}})$ in this case. For the convergence of the corresponding self-normalized change point test see Appendix \ref{appA}.

The test statistic $T_n(\tau_1, \tau_2)$ is designed for the detection of a single change-point.
An extension of the testing procedure that allows for multiple change-points is possible by adapting  Shao's testing procedure which takes this problem into consideration (see  \citep{Shao2011}). 
For convenience, we describe the construction of the modified test statistic in the case of two change-points. 
The general idea consists in dividing the sample given by $X_1, \ldots, X_n$ according to the pair $(k_1, k_2)$ of potential change-point locations and to compute the original test statistic with respect to the subsamples $X_1, \ldots, X_{k_2}$ and $X_{k_1+1}, \ldots, X_{n}$. We reject the hypothesis for large values of the sum of the corresponding single statistics.

For $\varepsilon\in (0, \tau_2-\tau_1)$ define
$T_n(\tau_1, \tau_2, \varepsilon):=\sup_{(k_1, k_2)\in \Omega_n(\tau_1, \tau_2, \varepsilon)}
\left|G_n(k_1, k_2)\right|$,
where $ \Omega_n(\tau_1, \tau_2, \varepsilon):=\left\{(k_1, k_2): \lfloor n\tau_1\rfloor\leq k_1<k_2\leq \lfloor n\tau_2\rfloor, \ k_2-k_1\geq \lfloor n\varepsilon \rfloor \right\}$
and
\begin{align*}
G_n(k_1, k_2)&:=\frac{\left|\sum_{i=1}^{k_1}R_{i}^{(1)}-\frac{k_1}{k_2}\sum_{i=1}^{k_2}R_{i}^{(1)}\right|}{\bigg\{\frac{1}{n}\sum_{t=1}^{k_1} \left(S_{t}^{(1)}(1,k_1)\right)^2+\frac{1}{n}\sum_{t=k_1+1}^{k_2} \left(S_{t}^{(1)}(k_1+1,k_2)\right)^{2}\bigg\}^{1/2}}\\
&\quad +\frac{\left|\sum_{i=k_1+1}^{k_2}R_{i}^{(2)}-\frac{k_2-k_1}{n-k_1}\sum_{i=k_1+1}^{n}R_{i}^{(2)}\right|}{\bigg\{\frac{1}{n}\sum_{t=k_1+1}^{k_2} \left(S_{t}^{(2)}(k_1+1,k_2)\right)^{2}+\frac{1}{n}\sum_{t=k_2+1}^{n} \left(S_{t}^{(2)}(k_1+1,n)\right)^{2}\bigg\}^{1/2}},
\end{align*}
where
\begin{align*}
&R_{i}^{(1)}:=\sum\limits_{j=1}^{k_2}1_{\left\{X_j\leq X_i\right\}}, \quad R_{i}^{(2)}.=\sum\limits_{j=k_1+1}^{n}1_{\left\{X_j\leq X_i\right\}},\\
& S_{t}^{(h)}(j, k):=\sum\limits_{i=j}^t\left(R_{i}^{(h)}-\bar{R}^{(h)}_{j, k}\right)\ \ \text{with } \ \bar{R}^{(h)}_{j, k}:=\frac{1}{k-j+1}\sum\limits_{t=j}^kR_t^{(h)}.
\end{align*}
The distribution of the test statistic converges to a limit $T(r, \tau_1, \tau_2, \varepsilon)$ (see Appendix \ref{appB}), so subsampling can be applied. The critical values corresponding to the asymptotic distribution of the  test statistic are reported in  Table \ref{cv_SN-Wilcoxon_2_cp}.  
\begin{table}[htbp]
\caption{Simulated critical values for the distribution of $T(1,  \tau_1, \tau_2, \varepsilon)$ when  $\left[\tau_1, \tau_2\right]=\left[0.15, 0.85\right]$ and  $\varepsilon =0.15$. The sample size is  $1 000$, the number of
replications is $10, 000$.} 		
\label{cv_SN-Wilcoxon_2_cp}	
\begin{tabular}{l r r r r }
& & 10\% &	 5\% 	 & 1\% 	\\	
\hline			
$H=0.501$ & & 17.79 &  19.76 &  24.13 \\ 
$H=0.6$ & & 19.80 & 22.38 & 27.68 \\ 
$H=0.7$ & & 22.08 & 24.95 & 30.46 \\ 
$H=0.8$ & & 24.24 & 27.61 & 34.04 \\ 
$H=0.9$ & & 26.50 & 30.11 & 37.78 \\ 
$H=0.999$ & & 28.28 & 32.32 & 41.24 \\ 
 \end{tabular}
\end{table}

\subsection{Data Examples}

We will revisit some data sets which have been analyzed before in the literature. We will use the self-normalized Wilcoxon change-point test combined with subsampling and compare our findings to the conclusions of other authors.

\begin{figure}[ht]
\includegraphics[scale=0.66]{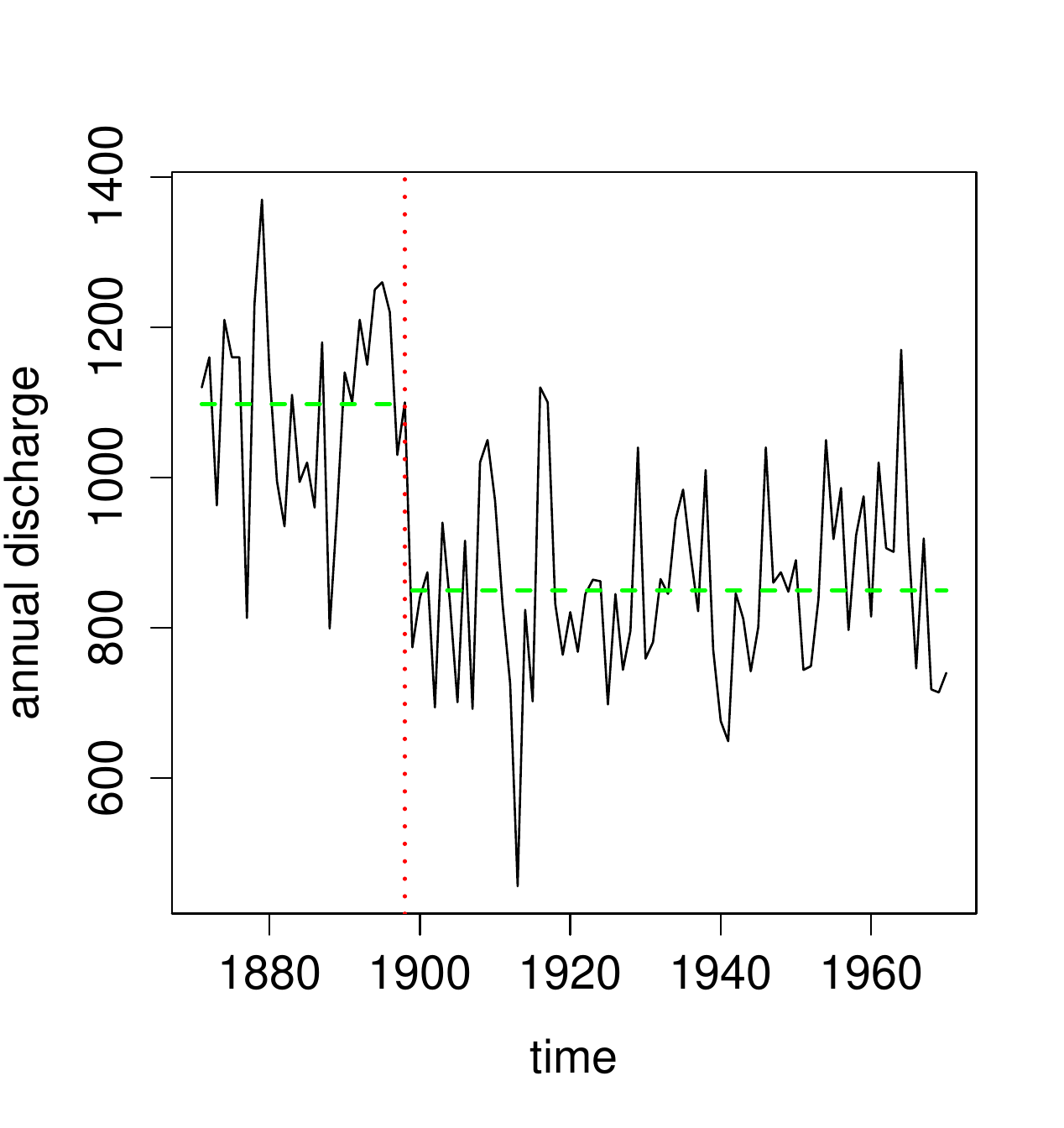}
\caption{Measurements of the annual discharge of the river Nile at  Aswan in $10^8$ $m^3$ for the years 1871-1970.  The dotted line indicates the location of the
change-point; the dashed lines designate the sample means for the pre-break and post-break samples.}
\label{Nile}
\end{figure}

The plot in Figure \ref{Nile} depicts 
 the annual volume of
discharge from the Nile river at Aswan in $10^8$ $m^3$ for the years 1871 to 1970.
The data set has been analyzed for the detection of a change-point by 
numerous authors under differing assumptions concerning the data generating random process and by usage of diverse methods.
Amongst others, Cobb \citep{Cobb1978}, MacNeill, Tang and Jandhyala \citep{MacNeill1991}, Wu and Zhao \citep{WuZhao2007} and Shao \citep{Shao2011} provided statistically significant evidence for a decrease of the Nile's annual discharge 	towards the end of the 19th century. The  construction of the Aswan Low Dam between 1898 and 1902  serves as a popular explanation for an abrupt change in the data.
%Yet, Cobb gave another explanation for the decrease in water volume by citing rainfall records which suggest a decline of tropical rainfall at that time.

The value of the self-normalized Wilcoxon test statistic computed with respect to the data is given by $T_n(\tau_1, \tau_2)=13.48729$. For a level of significance of $5\%$, the self-normalized Wilcoxon change-point test  rejects the hypothesis for every possible value of $H\in \left(\frac{1}{2}, 1\right)$. 
%If the nominal value of $H$ is smaller than  $0.8$, this would even yield a p-value less than $1\%$. 
Furthermore, we approximate the distribution of the self-normalized Wilcoxon test statistic by the sampling window method with block size $l=\lfloor \sqrt{n}\rfloor=10$. The subsampling-based test decision also  indicates the existence of a change-point in the mean of the data, even if we consider the $99\%$-quantile of $\hat{F}_{l, n}$.

In particular, 
previous analysis of the Nile data done by Wu and Zhao \citep{WuZhao2007} and Balke \citep{Balke1993} suggests that the change in the discharge volume occurred in 1899. We applied the self-normalized Wilcoxon test and the sampling window method to the corresponding pre-break and post-break samples. 
Neither of these two approaches leads to rejection of the hypothesis, so that it seems reasonable to consider  both samples as stationary. At this point, it is interesting to note that, based on the whole sample,  local Whittle estimation with bandwidth parameter $m = \lfloor n^{2/3}\rfloor$  suggests the existence of long range dependence characterized by an Hurst parameter $\hat{H} = 0.962$, whereas the estimates for the  pre-break and post-break samples given by $\hat{H}_1= 0.517$ and $\hat{H}_2=0.5$, respectively, should be considered as indication of short range dependent data. In this regard,  our findings support the conjecture of spurious long memory caused by a change-point and therefore coincide with the results of Shao \citep{Shao2011}.
\ignore{
Furthermore, the value of the test statistic,  which is constructed for the detection of two changes,  is given by $T_n(\tau_1, \tau_2, \varepsilon)=27.09477$ when $\varepsilon =\tau_1=1-\tau_2=0.15$.
Under the assumption of short range dependence,  a comparison with the critical values reported in Table \ref{cv_SN-Wilcoxon_2_cp} implies a test decision in favor of the alternative.
Applying the sampling window method with respect to this test statistic  yields a rejection of the hypothesis for any  choice of  block length $l\in\{\lfloor n^{\gamma}\rfloor|\ \gamma=0.6, 0.7, 0.8, 0.9\}=\{15, 25, 39, 63 \}$ based on comparison of  $T_n(\tau_1, \tau_2, \varepsilon)$ with the $99\%$-quantile of the corresponding sampling distribution $\hat{F}_{l, n}$. However, these results do not necessarily provide evidence for the assertion of two changes in the mean, since  a single distinct change-point in the data may easily lead to a rejection of the hypothesis of no change.}

\begin{figure}[ht]
\includegraphics[scale=0.66]{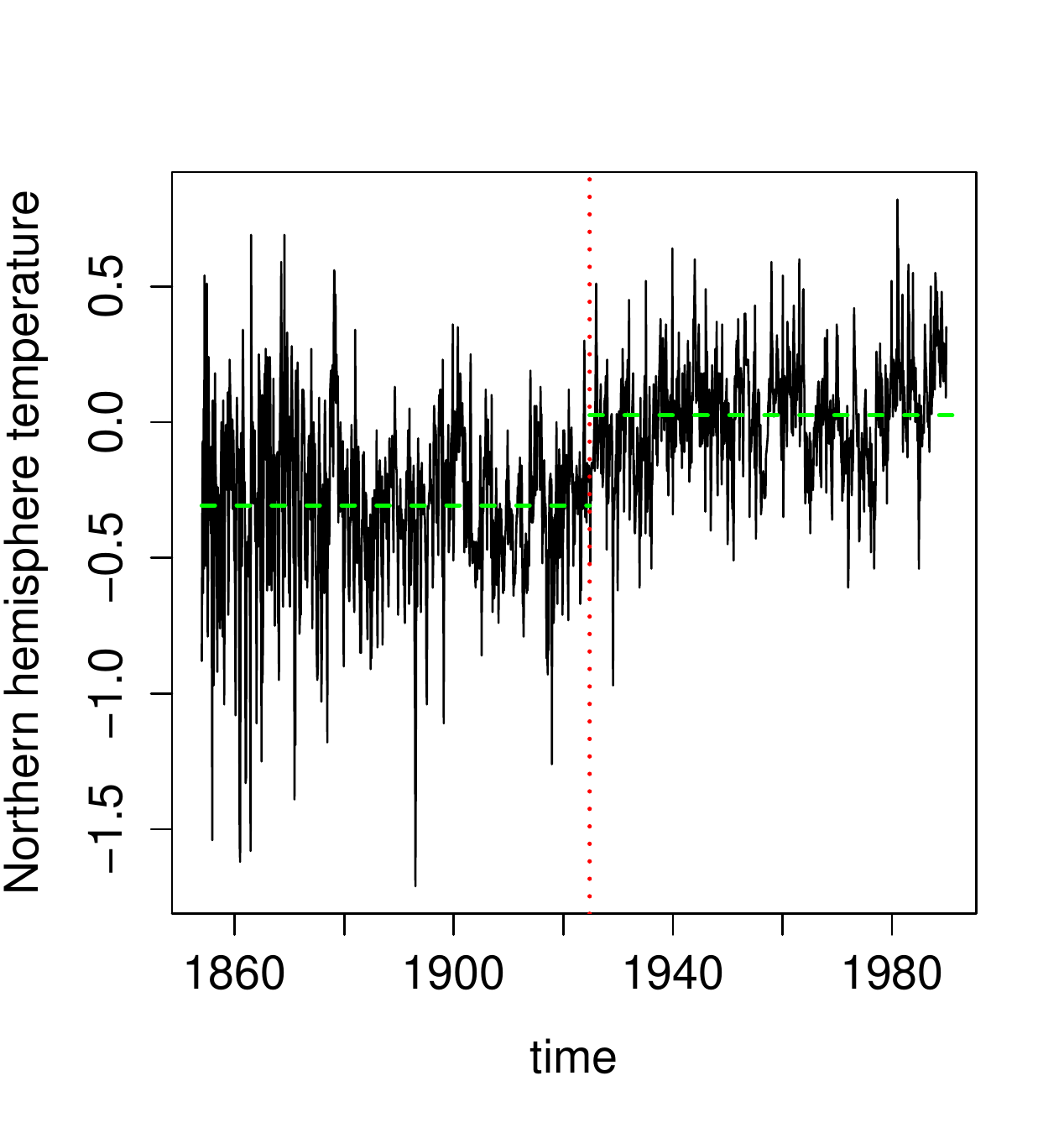}
\caption{Monthly temperature of the northern hemisphere for the years 1854-1989 from the data base held at the Climate Research Unit of the University of East Anglia, Norwich, England. The temperature anomalies (in degrees C) are calculated with respect to the reference period 1950-1979.  The dotted line indicates the location of the potential change-point; the dashed lines designate the sample means for the pre-break and post-break samples.}
\label{NemiTemp}
\end{figure}

The second data set consists of the seasonally adjusted monthly deviations of the  temperature (degrees C) for the northern hemisphere during the years 1854 to 1989 from the monthly averages over the period 1950 to 1979. The data results from spatial averaging of temperatures measured over land and sea. At first sight, the plot in Figure \ref{NemiTemp} may suggest an increasing trend as well as an abrupt change of the temperature deviations. Statistical evidence for a positive deterministic trend  implies affirmation of the conjecture that there has  been global warming during the last decades.

In scientific discourse, the question of whether the Northern hemisphere temperature  data acts  as an  indicator for global warming of the atmosphere is a controversial issue. Deo and Hurvich \citep{DeoHurvich1998} provided  some indication for global warming by fitting a linear trend to the data.  Beran and Feng \citep{BeranFeng2002} considered a more general stochastic model by the assumption of so-called semiparametric fractional autoregressive  (SEMIFAR) processes. Their method did not deliver sufficient statistical evidence for a deterministic trend. Wang \citep{Wang2007} applied another  method for the detection of gradual change to the global temperature data and did not detect an increasing trend , either. Nonetheless,   he offers an alternative explanation  for the occurrence of a trend-like behavior by  pointing out that it may have been generated by  stationary long range dependent processes. In contrast, it is shown in Shao \citep{Shao2011} that the existence of a change-point in the mean yields yet another explanation for the performance of the data.

The value of the self-normalized Wilcoxon test statistic computed with respect to the data is given by $T_n(\tau_1, \tau_2)=18.98636$. Consequently, the self-normalized Wilcoxon change-point test would  reject the hypothesis for every possible value of $H\in \left(\frac{1}{2}, 1\right)$ at a level of significance of $1\%$. In addition, an application of the sampling window method with respect to the self-normalized Wilcoxon test statistic based on comparison of  $T_n(\tau_1, \tau_2)$ with the $99\%$-quantile of the sampling distribution $\hat{F}_{l, n}$ yields a test decision in favor of the alternative hypothesis for any  choice of the block length $l\in\{\lfloor n^{\gamma}\rfloor|\ \gamma=0.3, 0.4, \ldots, 0.9\}=\{9, 19, 40, 84, 177, 371, 778\}$. All in all, both testing procedures  provide strong evidence for the existence of a change in the mean.

According to Shao \citep{Shao2011} the change-point  is located around October 1924. Based on the whole sample  local Whittle estimation with bandwidth $m=\lfloor n^{2/3}\rfloor$
provides an estimator $\hat{H}=0.811$. The estimated Hurst parameters for the pre-break  and post-break sample are $\hat{H}_1=0.597$ and $\hat{H}_2=0.88$, respectively. Neither of both testing procedures, i.e. subsampling with respect to the self-normalized Wilcoxon test statistic and comparison of the value of $T_n(\tau_1, \tau_2)$ with the corresponding critical values of its limit distribution, provides evidence for another change-point in the pre-break or post-break sample.

Moreover, computation of the test statistic  that allows for two  change-point locations yields $T_n(\tau_1, \tau_2, \varepsilon)=17.88404$ (for $\tau_1=1-\tau_2=\varepsilon=0.15$), i.e. if compared to the  values in Table \ref{cv_SN-Wilcoxon_2_cp}, the test statistic  only surpasses the critical value corresponding to  $H=0.501$ and a significance level of $10\%$, but  does not exceed any of the other values. Subsampling with respect to the test statistic $T_n(\tau_1, \tau_2, \varepsilon)$ does not support the conjecture of two changes, either. In fact, subsampling leads to a rejection of the hypothesis when the block length equals $l =\lfloor n^{0.7}\rfloor = 177$ (based on a comparison of  $T_n(\tau_1, \tau_2, \varepsilon)$ with the $95\%$-quantile of the corresponding sampling distribution $\hat{F}_{l, n}$), but yields a test decision in favor of the hypothesis for block lengths $l\in\{\lfloor n^{\gamma}\rfloor|\ \gamma=0.5, 0.6, 0.8, 0.9\}=\{40, 84, 371, 778\}$  and for comparison with the $90\%$-quantile of  $\hat{F}_{l, n}$.

Therefore, it seems safe to conclude that the appearance of long memory in the post-break sample is not caused by  another change-point in the mean. The pronounced difference between the local Whittle estimators $\hat{H}_1$ and $\hat{H}_2$ suggests  a change in the dependence structure of the times series. Another explanation might be a gradual change of the temperature in the post-break period. We conjecture that our test has only low power in the case of a gradual change, because the denominator of our self-normalized test statistic is inflated as the ranks systematically deviate from the mean rank of the first and second part. When using subsampling, the trend also appears  in subsamples so that we fail to approximate the distribution under the hypothesis.

As  pointed out by one of the referees, the Northern hemisphere temperature data  does  not seem to be second-order stationary; the variance in the first part of the time series seems to be higher. Such a change in variance should also result in a loss of power. The reason is that the ranks in the part with the higher variance are more extreme, so that the distance to the mean rank of this part is larger. This leads to a higher value of the denominator of our self-normalized test statistic and consequently to a lower value of the ratio.

The third data set consists of the arrival rate of Ethernet data (bytes per 10 milliseconds) from a local area network (LAN) measured at Bellcore Research and Engineering Center in 1989. For more information on the LAN traffic monitoring we refer to Leland and Wilson \citep{LelandWilson1991} and Beran \citep{Beran1994}. Figure \ref{ethernet} reveals that the observations are strongly right-skewed. As the self-normalized Wilcoxon test is based on ranks, we do not expect that this will affect our analysis.

\begin{figure}[ht]
\includegraphics[scale=0.6]{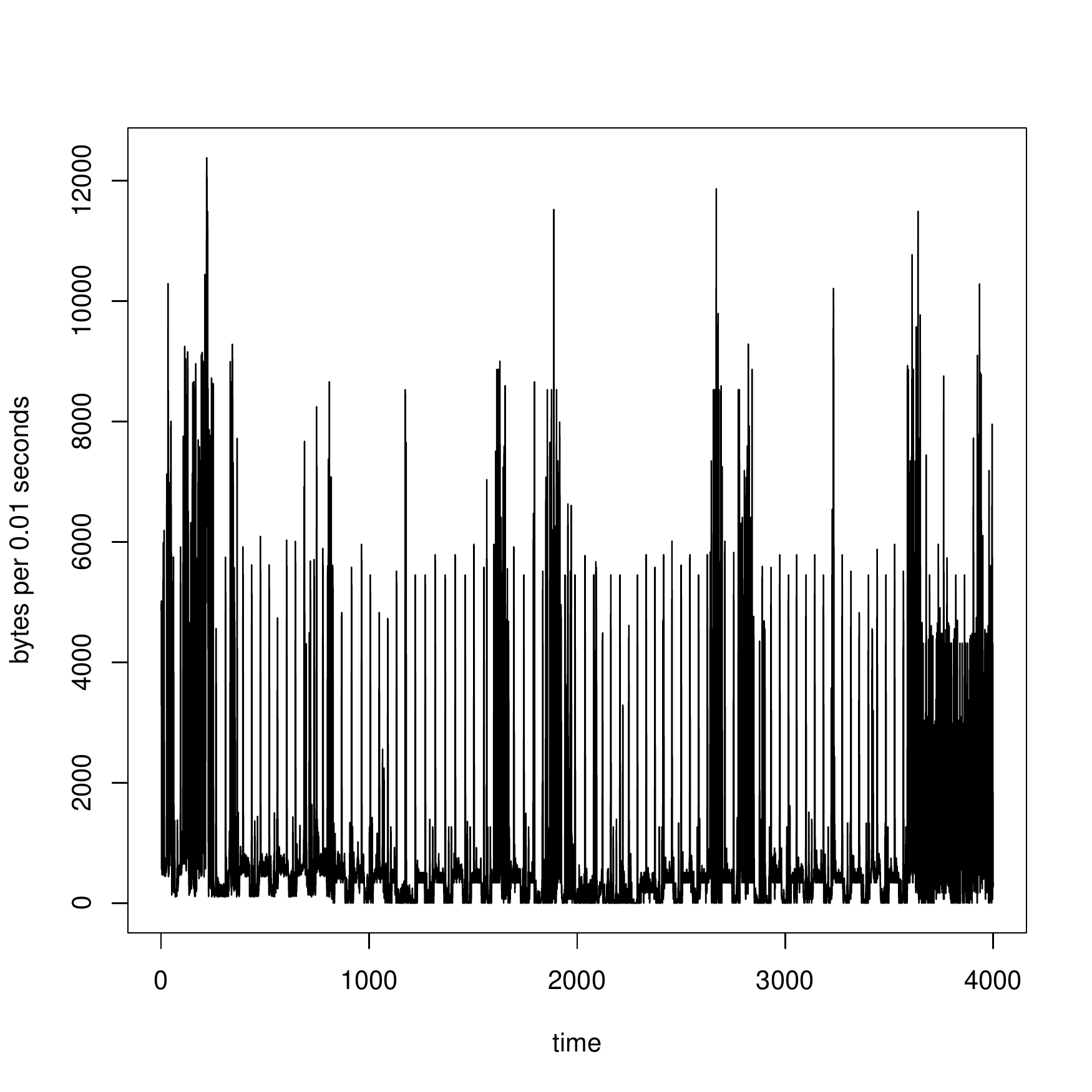}
\caption{Ethernet traffic in bytes per 10 milliseconds from a LAN measured at Bellcore Research Engineering Center.}
\label{ethernet}
\end{figure}

Coulon, Chabert and Swami \citep{Coulon2009} examined this data set for change-points before. The method proposed in their paper is based on the assumption that a FARIMA model holds for segments of the data. The number of different sections and the location of the change-points are chosen by a model selection criterion. The algorithm proposed by Coulon et al. \citep{Coulon2009} detects multiple changes in the parameters of the corresponding FARIMA time series.

In contrast, an application of the self-normalized Wilcoxon change-point test does not provide evidence for a change-point in the mean: the value of the  test statistic is given by $T_n(\tau_1, \tau_2)=3.270726$, i.e. even for a level of significance
of 10\%, the self-normalized Wilcoxon change-point test does not reject the hypothesis for any value $H\in \left(\frac{1}{2}, 1\right)$. Furthermore,  subsampling with respect to the self-normalized Wilcoxon test statistic does not lead to a rejection of the hypothesis , either (for any  choice of block length $l\in\{\lfloor n^{\gamma}\rfloor|\ \gamma=0.3, 0.4, \ldots, 0.9\}=\{12, 27, 63, 144, 332, 761, 1745\}$  and  for comparison with the $90\%$-quantile of  the corresponding sampling distribution $\hat{F}_{l, n}$). 

Taking into consideration that the data set contains ties  (the value $0$ appears several times), we also applied the self-normalized Wilcoxon test statistic  based on the modified ranks $\tilde{R}_i$ and used subsampling with respect to this statistic.
Both approaches did not lead to a rejection of the hypothesis.

An application of the test statistic  constructed for the detection of two changes  yields a value of $T_n(\tau_1, \tau_2, \varepsilon)=15.24527$ when $\varepsilon =\tau_1=1-\tau_2=0.15$.
Clearly, this does not lead to a rejection of the hypothesis for any value of the parameter $H$. In addition, subsampling
based on comparison of  $T_n(\tau_1, \tau_2, \varepsilon)$ with the $90\%$-quantile of the corresponding sampling distribution $\hat{F}_{l, n}$ does not provide evidence for the assertion of multiple changes for any block lenght  $l\in\{\lfloor n^{\gamma}\rfloor|\ \gamma= 0.5, 0.6, 0.7, 0.8\}=\{63, 144, 332, 761\}$ in the data,  either.

These results do not coincide with the analysis of the previous authors. On the one hand this may be due to the fact that the applied methods differ considerably from the testing procedures applied before. On the other hand, the change-point estimation algorithm proposed in Coulon, Chabert and Swami \citep{Coulon2009} is not robust to skewness or heavy-tailed distributions and decisively relies  on the assumption of FARIMA time series. However, this seems to contradict observations made by   Bhansali and Kokoszka \citep{BhansaliKokoszka2001} as well as Taqqu and Teverovsky \citep{TaqquTeverovsky1997} who stress that the  model that fits the   Ethernet traffic data is very unlikely to be FARIMA. 

 Estimation of the Hurst parameter by the local Whittle procedure  with bandwidth parameter $m = \lfloor n^{2/3}\rfloor$  yields an estimate of $\hat{H} = 0.845$ and therefore indicates the existence of long range dependence. This is consistent with the results of   Leland et al. \citep{Leland1994}
 and Taqqu and Teverovsky \citep{TaqquTeverovsky1997}.
 
In the three data examples, we find that the results obtained by subsampling and by parameter estimation are in good accordance with each other. The methods take into account long range dependence or heavy tails, but still detect a change in  location in the first two examples. For the third data example our analysis supports the hypothesis of stationarity.

\section{Simulations}\label{sec:simu}

We will now investigate the finite sample performance of the subsampling procedure  with respect to  the self-normalized Wilcoxon test and with respect to the classical Wilcoxon change-point test. Moreover, we will compare these results to the performance of the tests when the test decision is based on critical values obtained from the asymptotic distribution of the test statistic.

For this purpose, we consider subordinated Gaussian 
time series $\left(X_n\right)_{n\in \mathbb{N}}$, $X_n=G(\xi_n)$, where $\left(\xi_n\right)_{n\in \mathbb{N}}$
is fractional Gaussian noise (introduced in  Examples \ref{exa_fGn} and \ref{exa_fGn2}) with Hurst parameter $H\in\{0.6, 0.7, 0.8, 0.9\}$ and covariance function
\begin{align*}
\gamma(k)\sim k^{-D}\left(1-\frac{D}{2}\right)\left(1-D\right),
\end{align*}
where $D=2-2H$. Initially, we take $G(t)=t$, so that  $\left(X_n\right)_{n\in \mathbb{N}}$ has normal marginal distributions.
We also consider
the transformation 
\begin{align*}
G(t)=\left(\frac{\beta k^2}{(\beta -1)^2(\beta-2)}\right)^{-\frac{1}{2}}\left(k(\Phi(t))^{-\frac{1}{\beta}} -\frac{\beta k}{\beta -1}\right) 
\end{align*}
(with $\Phi$   denoting the standard normal distribution function) so as to generate Pareto-distributed data with parameters $k, \beta>0$ (referred to as Pareto($\beta$, $k$)).
In both cases, the  Hermite rank $r$ of $1_{\left\{G(\xi_i)\leq x\right\}}-F(x), x\in \mathbb{R}$, equals $r=1$ and 
\begin{align*}
 \left|\int_{\mathbb{R}}J_1(x)dF(x)\right|=\frac{1}{2\sqrt{\pi}};
\end{align*}
see \citep{DehlingRoochTaqqu2013a}.

Under the  above conditions, the critical values  of the asymptotic distribution of the self-normalized Wilcoxon test statistic are reported in Table 2 in  \citep{Betken2014}. 
The limit of the Wilcoxon change-point test statistic can be found in \citep{DehlingRoochTaqqu2013a},
the corresponding critical values can be taken from Table 1 in  \citep{Betken2014}. 

The  frequencies of rejections of both testing procedures are reported in Table \ref{SN-Wilcoxon(fGn)} and Table \ref{SN-Wilcoxon(Pareto)} for the self-normalized Wilcoxon change-point test and in Table \ref{Wilcoxon(fGn)} and Table \ref{Wilcoxon(Pareto)} for the classical Wilcoxon test (without self-normalization). The calculations are based on $5,000$ realizations of  time series with sample size $n=300$ and $n=500$. We have chosen  block lengths $l=l_n=\lfloor n^{\gamma}\rfloor$ with  $\gamma\in \left\{0.4, 0.5, 0.6\right\}$. As level of significance we chose $5\%$, i.e. we compare the values of the test statistic with the corresponding critical values of its asymptotic distribution and the corresponding quantile of the empirical distribution function $\hat{F}_{l, n}$, respectively.
 
For the usual testing procedures the estimation of the Hermite rank $r$, the  slowly varying function $L_{\gamma}$ and the integral $\int J_{1}(x)dF(x)$ is neglected. Yet, for every simulated time series we estimate the Hurst parameter $H$ by the local Whittle  estimator $\hat{H}$ proposed in K\"{u}nsch \cite{Kuensch1987}. This estimator is based on an approximation of the spectral density by the periodogram at the Fourier frequencies. It depends  on the spectral bandwidth parameter $m=m(n)$ which denotes the number of Fourier  frequencies used for the estimation.
If the  bandwidth $m$ satisfies  $\frac{1}{m}+\frac{m}{n}\longrightarrow 0$ as $n\longrightarrow \infty$, the local Whittle estimator is a consistent estimator for $H$; see Robinson \cite{Robinson1995}. For convenience we always choose $m = \lfloor n^{2/3}\rfloor$ in this article.
The  critical values corresponding to the estimated values of $H$ are determined by  linear interpolation.

Under the alternative $\operatorname{\mathbf{A}}$ we analyze the power of the testing procedures (the frequency of rejection) by considering different choices for the height of the level shift (denoted by $h$) and the location $[n\tau]$ of the change-point. In the tables the columns  that are superscribed by \enquote{$h=0$} correspond to the frequency of a type 1 error, i.e. the rejection rate under the hypothesis $\operatorname{\mathbf{H}}$.

For the self-normalized Wilcoxon change-point test (based on the asymptotic distribution), the empirical size almost equals the level of significance of $5\%$ for normally distributed data (see Table \ref{SN-Wilcoxon(fGn)}).  The  sampling window method yields rejection rates that slightly exceed this level. For Pareto($3$, $1$) time series both testing procedures lead to similar results and tend to reject the hypothesis too often when there is no change. With regard to the empirical power, it is notable that for fractional Gaussian noise time series the sampling window method yields considerably better power than the test based on asymptotic critical values.  If Pareto($3$, $1$)-distributed time series are considered, the empirical power of the subsampling procedure is still better  than the empirical power that results from using asymptotic critical values. However, in this case, the deviation  of the rejection rates is rather small. While the empirical size is not much affected by the Hurst parameter $H$, the empirical power is lower for $H=0.8, 0.9$.

Considering the classical Wilcoxon test (without self-normalization), it is notable that for both procedures the empirical size is in most cases not close to the nominal level of significance ($5\%$), ranging from $1.1\%$ to $20.8\%$ using subsampling and from $2.6\%$ to $36.0\%$ using asymptotic critical values. In general, the sampling window method becomes more conservative for higher values of the Hurst parameter $H$, while the test based on the asymptotic distribution becomes more liberal. Under the alternative, the usual application of the Wilcoxon test yields better power than the sampling window method, especially for high values of $H$. It should be emphasized that this comparison is problematic because the rejection frequencies under the hypothesis differ.

We conclude that the self-normalized Wilcoxon change-point test is more reliable than the classical change-point test. The reason might be that in the scaling of the classical test, the estimator $\hat{H}$ of the Hurst parameter enters as a power of the sample size $n$. Thus, a small error in this estimation might lead to a large error in the value of the test statistic. By using the sampling window method for the self-normalized version, we avoid the estimation of unknown parameters so that the performance is similar to the performance of the classical testing procedure which compares the values of the test statistic with the corresponding critical values.

Note that  in most cases covered by our simulations the choice of the block length for the subsampling procedure does not have a big impact on  the frequency of a type 1 error. Considering the self-normalized Wilcoxon change-point test, an increase of the block length tends to go along with a decrease in power, especially for big values of the Hurst parameter $H$ and Pareto-distributed random variables. For smaller values of $H$ the effect is not pronounced. We recommend using a block length $\lfloor n^{0.4}\rfloor$ or $\lfloor n^{0.5}\rfloor$ for the self-normalized change-point test as the choice $l=\lfloor n^{0.6}\rfloor$ implies worse properties in most cases.

An application of the subsampling testing procedure to the classical (non-self-normalized) Wilcoxon test for different choices of the block length shows the opposite effect on the rejection rate under the alternative: an increase of the block length results in a higher  frequency of rejections. Here, the block length $\lfloor n^{0.6}\rfloor$ leads to better results in many cases. However, we recommend to not use this test, but to self-normalize the test statistic instead.

An alternative way of choosing the block length would be to apply the data-driven block selection rule proposed by G\"{o}tze and Ra\u{c}kauskas \citep{GoetzeRackauskas2001} and Bickel and Sakov \citep{BickelSakov2008}.
Although the algorithm had originally been implemented for applications of the $m$-out-of-$n$ bootstrap to independent and identically   distributed data, it also lead to satisfactory simulation results in applications to long range dependent time series (see \citep{Jachetal2012}). Another general approach to the selection of the block size in the context of hypothesis testing  is given by Algorithm 9.4.2  in Politis, Romano and Wolf \citep{PolitisRomanoWolf1999}.

\begin{landscape}
\begin{table}[htbp]
\caption{Rejection rates of the self-normalized Wilcoxon change-point test obtained by subsampling (left) with block length $l = \lfloor n^{\gamma}\rfloor$, $\gamma \in \left\{0.4, 0.5, 0.6\right\}$, and by comparison with asymptotic critical values (right) for fractional Gaussian noise   of length $n$ with Hurst parameter $H$.}
\label{SN-Wilcoxon(fGn)}
\begin{tabular}{c c c c c c c c c c c c c c c c c c c c c}
 & & & \multicolumn{7}{l}{sampling window method}
 & &&&
\multicolumn{5}{l}{asymptotic distribution}\\
 \cline{4-12}  \cline{14-20}\\ 
 & &  & & & & &\multicolumn{2}{c}{$\tau = 0.25$} & &\multicolumn{2}{c}{$\tau = 0.5$} & & & & \multicolumn{2}{c}{$\tau = 0.25$} & & \multicolumn{2}{c}{$\tau = 0.5$} \\
 \cline{8-9}  \cline{11-12}     \cline{16-17} \cline{19-20}\\
fGn & $n$ & & $l$  &  & $h = 0$ &  & $h = 0.5$ & $h = 1$ &  & $h = 0.5$ & $h = 1$ &  & $h=0$ &  & $h = 0.5$ & $h = 1$ &  & $h = 0.5$ & $h = 1$ \\ 
 \hline
   $H = 0.6$ & 300 & &  9 &  & 0.041 &  & 0.263 & 0.700 &  & 0.502 & 0.952 \\ 
&  & & 17 & & 0.064 &  & 0.313 & 0.742 &  & 0.570 & 0.964 &  & 0.044 &  & 0.209 & 0.521 &  & 0.424 & 0.861 \\ 
 & & &  30 &  & 0.070 &  & 0.322 & 0.705 &  & 0.555 & 0.943 \\ 
 & 500 & &  12 &  & 0.053 &  & 0.396 & 0.859 &  & 0.697 & 0.994 \\ 
 &  & & 22 &  & 0.060 &  & 0.421 & 0.861 &  & 0.720 & 0.995 &  & 0.049 &  & 0.303 & 0.687 &  & 0.577 & 0.958 \\ 
  &  & &  41 &  & 0.069 &  & 0.411 & 0.829 &  & 0.697 & 0.991 \\ 
    $H = 0.7$ & 300 & & 9 &  & 0.057 &  & 0.155 & 0.412 &  & 0.291 & 0.759 \\ 
 &  & & 17  & & 0.070 &  & 0.171 & 0.423 &  & 0.313 & 0.763 &  & 0.053 &  & 0.108 & 0.268 &  & 0.228 & 0.611 \\ 
  & & &  30 & & 0.077 &  & 0.177 & 0.403 &  & 0.314 & 0.737 \\ 
  & 500 & &  12 &  & 0.056 &  & 0.183 & 0.513 &  & 0.382 & 0.856 \\ 
 & & & 22 & & 0.059 &  & 0.193 & 0.508 &  & 0.382 & 0.854 &  & 0.048 &  & 0.133 & 0.359 &  & 0.302 & 0.730 \\ 
  &  & &  41 &  & 0.065 &  & 0.192 & 0.476 &  & 0.387 & 0.819 \\ 
    $H = 0.8$ & 300 & & 9 &  & 0.070 &  & 0.126 & 0.251 &  & 0.223 & 0.526 \\ 
 & & & 17 &  & 0.067 &  & 0.117 & 0.234 &  & 0.208 & 0.494 &  & 0.048 &  & 0.081 & 0.144 &  & 0.141 & 0.362 \\ 
  &  & &  30 &  & 0.073 &  & 0.114 & 0.218 &  & 0.201 & 0.466 \\
  & 500 & &  12 &  & 0.066 &  & 0.121 & 0.295 &  & 0.217 & 0.591 \\ 
 &  & & 22 & & 0.068 &  & 0.114 & 0.278 &  & 0.210 & 0.567 &  & 0.053 &  & 0.085 & 0.198 &  & 0.163 & 0.462 \\ 
  &  & & 41 &  & 0.069 &  & 0.119 & 0.257 &  & 0.205 & 0.532 \\ 
    $H = 0.9$ & 300 & &  9 &  & 0.093 &  & 0.126 & 0.208 &  & 0.209 & 0.462 \\ 
 & & & 17 & & 0.074 &  & 0.097 & 0.161 &  & 0.169 & 0.397 &  & 0.057 &  & 0.065 & 0.106 &  & 0.125 & 0.308 \\ 
  &  & &  30 &  & 0.073 &  & 0.095 & 0.145 &  & 0.165 & 0.367 \\ 
  & 500 & &  12 &  & 0.079 &  & 0.105 & 0.194 &  & 0.185 & 0.461 \\ 
 & & & 22 & & 0.067 &  & 0.091 & 0.166 &  & 0.162 & 0.416 &  & 0.051 &  & 0.068 & 0.120 &  & 0.128 & 0.350 \\ 
  & & &  41 &  & 0.063 &  & 0.087 & 0.146 &  & 0.152 & 0.391 \\ 
 \end{tabular}
\end{table}
\end{landscape}

\begin{landscape}
\begin{table}[htbp]
\caption{Rejection rates of the self-normalized Wilcoxon change-point test obtained by subsampling (left) with block length $l = \lfloor n^{\gamma}\rfloor$, $\gamma \in \left\{0.4, 0.5, 0.6\right\}$,   and by comparison with asymptotic critical values (right) for  Pareto($3$, $1$)-transformed fractional Gaussian noise  of length $n$ with Hurst parameter $H$.}
\label{SN-Wilcoxon(Pareto)}
\begin{tabular}{c c c c c c c c c c c c c c c c c c c c c}
 & & & \multicolumn{7}{l}{sampling window method}
 & &&&
\multicolumn{5}{l}{asymptotic distribution}\\
 \cline{4-12}  \cline{14-20}\\ 
 & &  & & & & &\multicolumn{2}{c}{$\tau = 0.25$} & &\multicolumn{2}{c}{$\tau = 0.5$} & & & & \multicolumn{2}{c}{$\tau = 0.25$} & & \multicolumn{2}{c}{$\tau = 0.5$} \\
 \cline{8-9}  \cline{11-12}     \cline{16-17} \cline{19-20}\\
Pareto(3, 1) & $n$ & & $l$  &  & $h = 0$ &  & $h = 0.5$ & $h = 1$ &  & $h = 0.5$ & $h = 1$ &  & $h=0$ &  & $h = 0.5$ & $h = 1$ &  & $h = 0.5$ & $h = 1$ \\  
 \hline
$H = 0.6$     & 300 &  & 9 &  & 0.041 &  & 0.847 & 0.977 &  & 0.990 & 1.000 \\ 
 & & & 17 & & 0.067 &  & 0.871 & 0.946 &  & 0.990 & 1.000 &  & 0.056 &  & 0.820 & 0.912 &  & 0.984 & 0.999 \\ 
  & & &  30 & & 0.070 &  & 0.831 & 0.946 &  & 0.979 & 1.000 \\ 
     & 500 &  & 12 & & 0.055 &  & 0.947 & 0.997 &  & 0.999 & 1.000 \\ 
 &  &  & 22 & & 0.066 &  & 0.946 & 0.994 &  & 0.999 & 1.000 &  & 0.061 &  & 0.920 & 0.970 &  & 0.996 & 1.000 \\ 
   & & &  41 &  & 0.071 &  & 0.921 & 0.976 &  & 0.996 & 1.000 \\ 
$H = 0.7$  &  300 & & 9 &  & 0.057 &  & 0.571 & 0.821 &  & 0.990 & 0.994 \\ 
 & & & 17& & 0.064 &  & 0.527 & 0.738 &  & 0.876 & 0.990 &  & 0.070 &  & 0.529 & 0.702 &  & 0.856 & 0.982 \\ 
  & & &  30 & & 0.077 &  & 0.527 & 0.738 &  & 0.842 & 0.975 \\ 
     & 500 &  & 12 & & 0.066 &  & 0.693 & 0.904 &  & 0.949 & 0.999 \\ 
 &  & & 22 &  & 0.068 &  & 0.684 & 0.893 &  & 0.942 & 0.998 &  & 0.076 &  & 0.663 & 0.820 &  & 0.940 & 0.995 \\ 
   & & &  41 &  & 0.072 &  & 0.632 & 0.838 &  & 0.921 & 0.994 \\ 
   $H = 0.8$  &  300 &  & 9 & & 0.070 &  & 0.355 & 0.574 &  & 0.703 & 0.931 \\ 
 & & & 17 & & 0.068 &  & 0.284 & 0.454 &  & 0.666 & 0.905 &  & 0.072 &  & 0.297 & 0.428 &  & 0.640 & 0.875 \\ 
   & & &  30 & & 0.073 &  & 0.284 & 0.454 &  & 0.633 & 0.857 \\ 
     & 500 &  & 12 &  & 0.064 &  & 0.401 & 0.609 &  & 0.738 & 0.948 \\ 
 &  & & 22 & & 0.063 &  & 0.379 & 0.581 &  & 0.714 & 0.933 &  & 0.069 &  & 0.369 & 0.510 &  & 0.715 & 0.920 \\ 
   & & & 41 &  & 0.064 &  & 0.345 & 0.509 &  & 0.688 & 0.903 \\
   $H = 0.9$  &  300 & & 9 &  & 0.093 &  & 0.253 & 0.396 &  & 0.597 & 0.832 \\ 
 &  & & 17 & & 0.071 &  & 0.168 & 0.254 &  & 0.532 & 0.772 &  & 0.073 &  & 0.165 & 0.236 &  & 0.499 & 0.738 \\ 
   & & &  30 & & 0.073 &  & 0.168 & 0.254 &  & 0.482 & 0.729 \\ 
     & 500 &  &12 & 
 & 0.073 &  & 0.256 & 0.405 &  & 0.585 & 0.839 \\ 
 &  & & 22 & & 0.064 &  & 0.219 & 0.340 &  & 0.547 & 0.802 &  & 0.068 &  & 0.199 & 0.296 &  & 0.529 & 0.782 \\ 
   & & &  41 &  & 0.065 &  & 0.190 & 0.296 &  & 0.503 & 0.762 \\ 
\end{tabular}
\end{table}
\end{landscape}

\begin{landscape}
\begin{table}[htbp]
\caption{Rejection rates of the classical Wilcoxon change-point test obtained by subsampling (left) with block length $l = \lfloor n^{\gamma}\rfloor$, $\gamma \in \left\{0.4, 0.5, 0.6\right\}$,   and by comparison with asymptotic critical values (right) for fractional Gaussian noise   of length $n$ with Hurst parameter $H$.}
\label{Wilcoxon(fGn)}
\begin{tabular}{c c c c c c c c c c c c c c c c c c c c c}
 & & & \multicolumn{7}{l}{sampling window method}
 & &&&
\multicolumn{5}{l}{asymptotic distribution}\\
 \cline{4-12}  \cline{14-20}\\ 
 & &  & & & & &\multicolumn{2}{c}{$\tau = 0.25$} & &\multicolumn{2}{c}{$\tau = 0.5$} & & & & \multicolumn{2}{c}{$\tau = 0.25$} & & \multicolumn{2}{c}{$\tau = 0.5$} \\
 \cline{8-9}  \cline{11-12}     \cline{16-17} \cline{19-20}\\
fGn & $n$ & & $l$  &  & $h = 0$ &  & $h = 0.5$ & $h = 1$ &  & $h = 0.5$ & $h = 1$ &  & $h=0$ &  & $h = 0.5$ & $h = 1$ &  & $h = 0.5$ & $h = 1$ \\ 
 \hline
 $H = 0.6$ & 300 & &  9 & &   0.066 &  & 0.20 & 0.232 &  &  0.386 & 0.591 \\ 
  &  & &17 & & 0.054 &  & 0.223 & 0.411 &  & 0.439 & 0.784 &  & 0.026 &  & 0.096 & 0.160 &  & 0.223 & 0.727 \\ 
 & & & 30 &  & 0.059 &  & 0.264 & 0.529 &  & 0.663 & 0.870 \\ 
  & 500 & & 12 &  & 0.063 &  & 0.285 & 0.436 &  & 0.569 & 0.856 \\ 
 &  & & 22 &  & 0.058 &  & 0.345 & 0.663 &  & 0.627 & 0.952 &  & 0.036 &  & 0.148 & 0.256 &  & 0.378 & 0.897 \\ 
  & & & 41&  & 0.062 &  & 0.397 & 0.789 &  & 0.683 & 0.975 \\ 
   $H = 0.7$ & 300 & & 9 & & 0.052 &  & 0.080 & 0.088 &  & 0.162 & 0.302 \\ 
&  & & 17 &  & 0.049 &  & 0.095 & 0.158 &  & 0.206 & 0.466 &  & 0.035 &  & 0.067 & 0.228 &  & 0.167 & 0.665 \\ 
 & & & 30 & & 0.051 &  & 0.120 & 0.227 &  & 0.267 & 0.593 \\ 
   & 500 & & 12&  & 0.042 &  & 0.104 & 0.153 &  & 0.249 & 0.539 \\ 
 &  & & 22 &   & 0.039 &  & 0.131 & 0.267 &  & 0.287 & 0.689 &  & 0.030 &  & 0.079 & 0.259 &  & 0.225 & 0.714 \\
  & & & 41 &  & 0.046 &  & 0.160 & 0.373 &  & 0.343 & 0.789 \\ 
   $H = 0.8$ & 300 & & 9 &  & 0.028 &  & 0.030 & 0.031 &  & 0.054 & 0.092 \\ 
 &  & & 17 &  & 0.029 &  & 0.038 & 0.048 &  & 0.075 & 0.179 &  & 0.077 &  & 0.153 & 0.421 &  & 0.245 & 0.673 \\ 
 & & & 30 &  & 0.034 &  & 0.057 & 0.088 &  & 0.070 & 0.272 \\ 
   & 500 & & 12 &  & 0.023 &  & 0.031 & 0.036 &  & 0.064 & 0.162 \\ 
 &  & & 22&   & 0.028 &  & 0.044 & 0.070 &  & 0.097 & 0.273 &  & 0.050 &  & 0.112 & 0.439 &  & 0.226 & 0.714 \\ 
  & & & 41 &  & 0.039 &  & 0.071 & 0.129 &  & 0.137 & 0.391 \\ 
   $H = 0.9$ & 300 & & 9 &  & 0.009 &  & 0.010 & 0.006 &  & 0.016 & 0.020 \\ 
 &  & &17&  & 0.009 &  & 0.014 & 0.009 &  & 0.021 & 0.060 &  & 0.36 &  & 0.484 & 0.739 &  & 0.524 & 0.830 \\ 
 & & & 30 &  & 0.015 &  & 0.029 & 0.028 &  & 0.011 & 0.153 \\ 
   & 500 & & 12& 
 & 0.008 &  & 0.006 & 0.003 &  & 0.015 & 0.026 \\ 
 &  & & 22 &  & 0.011 &  & 0.009 & 0.011 &  & 0.029 & 0.086 &  & 0.319 &  & 0.439 & 0.743 &  & 0.511 & 0.845 \\ 
  & & & 41 &  & 0.021 &  & 0.021 & 0.032 &  & 0.058 & 0.197 \\ 
  \end{tabular}
\end{table}
\end{landscape}

  \begin{landscape}
\begin{table}[htbp]
\caption{Rejection rates of the classical Wilcoxon change-point test obtained by subsampling (left) with block length $l = \lfloor n^{\gamma}\rfloor$, $\gamma \in \left\{0.4, 0.5, 0.6\right\}$,   and by comparison with asymptotic critical values (right) for  Pareto($3$, $1$)-transformed fractional Gaussian noise  of length $n$ with Hurst parameter $H$.}
\label{Wilcoxon(Pareto)}
\begin{tabular}{c c c c c c c c c c c c c c c c c c c c c}
 & & & \multicolumn{7}{l}{sampling window method}
 & &&&
\multicolumn{5}{l}{asymptotic distribution}\\
 \cline{4-12}  \cline{14-20}\\ 
 & &  & & & & &\multicolumn{2}{c}{$\tau = 0.25$} & &\multicolumn{2}{c}{$\tau = 0.5$} & & & & \multicolumn{2}{c}{$\tau = 0.25$} & & \multicolumn{2}{c}{$\tau = 0.5$} \\
 \cline{8-9}  \cline{11-12}     \cline{16-17} \cline{19-20}\\
Pareto(3, 1) & $n$ & & $l$  &  & $h = 0$ &  & $h = 0.5$ & $h = 1$ &  & $h = 0.5$ & $h = 1$ &  & $h=0$ &  & $h = 0.5$ & $h = 1$ &  & $h = 0.5$ & $h = 1$ \\ 
 \hline
  $H = 0.6$ & 300 & &  9 & & 0.170 &  & 0.949 & 0.742 &  & 0.991 & 0.923 \\ 
  &  & &  17 & & 0.130 &  & 0.963 & 0.861 &  & 0.996 & 0.991 &  & 0.108 &  & 0.938 & 0.985 &  & 0.998 & 1.000 \\ 
 &  & & 30 &  & 0.109 &  & 0.962 & 0.871 &  & 0.998 & 0.998 \\ 
   & 500 & &  12 &  & 0.163 &  & 0.991 & 0.916 &  & 1.000 & 0.993 \\ 
  &  & &  22 &  & 0.132 &  & 0.997 & 0.976 &  & 1.000 & 0.999 &  & 0.128 &  & 0.988 & 0.999 &  & 1.000 & 1.000 \\ 
    &  & &  41 &  & 0.114 &  & 0.997 & 0.989 &  & 1.000 & 1.000 \\ 
  $H = 0.7$ & 300 & &  9 &  & 0.224 &  & 0.785 & 0.568 &  & 0.939 & 0.796 \\ 
  &  & &  17 &  & 0.175 &  & 0.802 & 0.680 &  & 0.955 & 0.949 &  & 0.179 &  & 0.833 & 0.969 &  & 0.974 & 0.999 \\ 
  &  & &  30 & 
 & 0.140 &  & 0.789 & 0.708 &  & 0.959 & 0.976 \\ 
    & 500 & &  12 &  & 0.208 &  & 0.921 & 0.763 &  & 0.989 & 0.956 \\ 
  &  & &  22 &  & 0.167 &  & 0.931 & 0.862 &  & 0.992 & 0.996 &  & 0.191 &  & 0.940 & 0.994 &  & 0.996 & 1.000 \\ 
  &  & &  41 & & 0.143 &  & 0.925 & 0.891 &  & 0.994 & 0.998 \\
  $H = 0.8$ & 300  & &  9 &  & 0.203 &  & 0.508 & 0.326 &  & 0.743 & 0.565 \\ 
  &  & &  17 &  & 0.160 &  & 0.496 & 0.347 &  & 0.776 & 0.808 &  & 0.204 &  & 0.729 & 0.925 &  & 0.918 & 0.993 \\ 
  &  & &  30 &  & 0.137 &  & 0.484 & 0.364 &  & 0.791 & 0.881 \\ 
    & 500 & &  12 &  & 0.190 &  & 0.639 & 0.445 &  & 0.865 & 0.770 \\ 
  &  & &  22 & & 0.160 &  & 0.649 & 0.513 &  & 0.886 & 0.929 &  & 0.212 &  & 0.805 & 0.963 &  & 0.948 & 0.999 \\ 
  &  & &  41 &  & 0.137 &  & 0.626 & 0.556 &  & 0.890 & 0.961 \\ 
  $H = 0.9$ &  300 & &  9 & 
 & 0.128 &  & 0.150 & 0.077 &  & 0.320 & 0.336 \\ 
  &  & &  17 & & 0.097 &  & 0.128 & 0.071 &  & 0.403 & 0.550 &  & 0.309 &  & 0.712 & 0.901 &  & 0.848 & 0.966 \\ 
  &  & &  30 &  & 0.092 &  & 0.125 & 0.077 &  & 0.481 & 0.677 \\ 
    & 500 & & 12 &  & 0.112 &  & 0.159 & 0.089 &  & 0.402 & 0.436 \\
  &  & &   22 &  & 0.100 &  & 0.161 & 0.101 &  & 0.518 & 0.680 &  & 0.27 &  & 0.726 & 0.911 &  & 0.851 & 0.975 \\ 
  &  & & 41  &  & 0.095 &  & 0.170 & 0.106 &  & 0.571 & 0.771 \\ 
\end{tabular}
\end{table}
\end{landscape}

\newpage

\section{Proofs}\label{sec:proof}

\subsection{Auxiliary Results}

\begin{Lem}\label{lemA} Under Assumption \ref{dgp}, there is a constant  $K_D<\infty$, such that for all $x_1,\ldots,x_l\in\R$ with $\Var(\sum_{i=1}^lx_i\xi_i)=1$
\begin{equation*}
\sum_{i=1}^lx_i^2\leq K_D.
\end{equation*}
\end{Lem}

\begin{proof} Recall that we can rewrite the covariances as
\begin{equation*}
\gamma(k)=\int_{-\pi}^{\pi}\e^{ik\lambda}f(\lambda)d\lambda
\end{equation*}
and that the spectral density $f$ can be written as $f(\lambda)=L_f(|\lambda|)|\lambda|^{D-1}$. By our assumptions $L_f(x)\geq C_{\text{min}}$ for a constant $C_{\text{min}}>0$, so that we can conclude that
\begin{align*}
1=&\Var\Big(\sum_{i=1}^lx_i\xi_i\Big)=\sum\limits_{1\leq j, k\leq l} x_jx_k\gamma(j-k)\\
=&\sum\limits_{1\leq j, k\leq l} x_jx_k\int_{-\pi}^{\pi}\e^{i(j-k)\lambda}f(\lambda)d\lambda=\sum\limits_{1\leq j, k\leq l} x_jx_k\int_{-\pi}^{\pi}\e^{i(j-k)\lambda}L_f(|\lambda|)|\lambda|^{D-1}d\lambda\displaybreak[0]\\
=&2\int_{0}^{\pi}\!\sum\limits_{1\leq j, k\leq l} x_jx_k\e^{i(j-k)\lambda}L_f(\lambda)\lambda^{D-1}d\lambda\displaybreak[0]=2\int_{0}^{\pi}\bigg|\sum\limits_{j=1}^lx_j\e^{-ij\lambda}\bigg|^2L_f(\lambda)\lambda^{D-1}d\lambda\\
\geq& 2C_{\text{min}}\pi^{D-1} \int_{0}^{\pi}\bigg|\sum\limits_{j=1}^lx_j\e^{-ij\lambda}\bigg|^2d\lambda.
\end{align*}
We rewrite the integrand as
\begin{align*}
\bigg|\sum\limits_{j=1}^lx_j\e^{-ij\lambda}\bigg|^2
&=\sum\limits_{1\leq j, k\leq l}x_jx_k\e^{-ij\lambda}\e^{ik\lambda}=\sum\limits_{ j=1}^lx_j^2+\sum\limits_{ j\neq k}x_jx_k\e^{-i(j-k)\lambda}\displaybreak[0]\\
&=\sum\limits_{ j=1}^lx_j^2+\sum\limits_{ j < k}x_jx_k\left(\e^{-i(j-k)\lambda}+\e^{-i(k-j)\lambda}\right)\\
&=\sum\limits_{ j=1}^lx_j^2+2\sum\limits_{ j < k}x_jx_k\cos((k-j)\lambda)=\sum\limits_{1\leq j, k\leq l}x_jx_k\cos((k-j)\lambda).
\end{align*}
As a result, we have
\begin{align*}
 \int_{0}^{\pi}
\bigg|\sum\limits_{j=1}^lx_j\e^{-ij\lambda}\bigg|^2d\lambda
&= \int_{0}^{\pi}\sum\limits_{1\leq j, k\leq l} x_jx_k\cos((k-j)\lambda)
d\lambda\displaybreak[0]\\
&= \sum\limits_{1\leq j, k\leq l} x_jx_k\int_{0}^{\pi}\cos((k-j)\lambda)
d\lambda\displaybreak[0]\\
&= \sum\limits_{ j=1}^lx_j^2\int_{0}^{\pi}\cos(0)
d\lambda+\sum\limits_{ j\neq k}x_jx_k\int_{0}^{\pi}\cos((k-j)\lambda)
d\lambda\\
&=\pi  \sum\limits_{ j=1}^lx_j^2.
\end{align*}
All in all, this yields
\begin{equation*}
1=\Var\left(\sum_{i=1}^lx_i\xi_i\right)\geq 2 C_{\text{min}}	\pi^{D-1} \int_{0}^{\pi}
\bigg|\sum\limits_{j=1}^lx_j\e^{-ij\lambda}\bigg|^2d\lambda= 2 C_{\text{min}}	\pi^{D}\sum\limits_{ j=1}^lx_j^2.
\end{equation*}
Therefore, the statement of the lemma holds with $K_D=1/(2 C_{\text{min}}	\pi^{D})$.

\end{proof}

\begin{Lem}\label{lemB} Under Assumption \ref{dgp}, there are constants  $K'_D<\infty$ and $l_0\in\N$ such that 
\begin{equation*}
\bigg|\sum_{i=1}^lx_i\bigg|\leq K'_Dl^{D/2}
\end{equation*}
for all $l\geq l_0$ and $x_1,\ldots,x_l\in\R$ with $\Var\left(\sum_{i=1}^lx_i\xi_i\right)=1$.
\end{Lem}

\begin{proof} The statement of the proof is equivalent to the existence of a constant $C>0$, such that for all 
$x_1,\ldots,x_l\in\R$ with $\sum_{i=1}^lx_i=1$, we have 
\begin{equation*}
\Var\left(\sum_{i=1}^lx_i\xi_i\right)\geq Cl^{-D}.
\end{equation*}
Let  $x_1^\star,\ldots,x_l^\star\in\R$ with $\sum_{i=1}^lx_i^\star=1$ be the values that minimize $\Var\left(\sum_{i=1}^lx_i^\star\xi_i\right)$. Then $\hat{\mu}_{\xi}(\xi_1,\ldots,\xi_n):=\sum_{i=1}^lx_i^\star\xi_i$ is the best linear unbiased estimator for $\mu:=E(\xi_1)$.
For a process $(\zeta_n)_{n\in\N}$ with spectral density
\begin{equation*}
f_{\zeta}(x)=\frac{1}{2\pi}\left|1-e^{ix}\right|^{D-1},
\end{equation*}
we have 
\begin{equation*}
\Var\left(\hat{\mu}_{\zeta}(\zeta_1,\ldots,\zeta_n)\right)\geq C_1l^{-D}
\end{equation*}
for a constant $C_1>0$ by a Corollary  of Adenstedt \citep{Adenstedt1974} (see p. 1101). We rewrite the spectral density $f_{\zeta}$ of $(\zeta_n)_{n\in\N}$ with the help of the spectral density $f$ of $(\xi_n)_{n\in\N}$ as
\begin{equation*}
f_{\zeta}(x)=f(x)\frac{\left|1-e^{ix}\right|^{D-1}}{2\pi|x|^{D-1}L_f(x)}.
\end{equation*}
Note that the function $g$ with $g(x)=\frac{\left|1-e^{ix}\right|^{D-1}}{2\pi|x|^{D-1}L_f(x)}$ is bounded, as we assumed that $L_f$ is bounded away from $0$.  Hence, we have
\begin{equation*}
\Var\left(\hat{\mu}_{\xi}(\xi_1,\ldots,\xi_n)\right)\geq \frac{1}{g(0)}\Var\left(\hat{\mu}_{\zeta}(\zeta_1,\ldots,\zeta_n)\right)\geq Cl^{-D}
\end{equation*}
 for all $l\geq l_0$
by Lemma 4.4 in \citep{Adenstedt1974}.
\end{proof}

The next Lemma deals with the $\rho$-mixing coefficient, which is defined in the following way: Let $\mathcal{A},\mathcal{B}$ be two $\sigma$-fields. Then
\begin{equation*}
\rho(\mathcal{A},\mathcal{B}):=\sup \corr(X,Y),
\end{equation*}
where the supremum is taken over all $\mathcal{A}$-measurable random variables $X$ and all $\mathcal{B}$-measurable random variables $Y$. For details we recommend the book of Bradley \citep{bradley2007introduction}.

\begin{Lem}\label{LemC} Under Assumption \ref{dgp}, there are constants $C_1,C_2<\infty$ such that
\begin{multline*}
\rho(k,l):=\rho\big(\sigma(\xi_i, 1\leq i\leq l), \sigma(\xi_j, k+l+1\leq j\leq k+2l)\big)\\
\leq C_1\left(k/l\right)^{-D}L_{\gamma}(k)+C_2l^2k^{-D-1}\max\{L_{\gamma}(k),1\}
\end{multline*}
for all $k\in\N$ and all $l\in\{l_k,\ldots,k\}$.
\end{Lem}

\begin{proof} Kolmogorov and Rozanov \citep{kolmogorov1960strong}  proved that there exist real numbers $a_1, a_2, \ldots, a_l$, $b_1, b_2, \ldots, b_l$ such that
\begin{equation*}
\rho\big(\sigma(\xi_i, 1\leq i\leq l), \sigma(\xi_j, k+l+1\leq j\leq k+2l)\big)
=\Cov\Big(\sum\limits_{i=1}^l a_i\xi_i, \sum\limits_{j=1}^l b_j\xi_{k+l+j}\Big)
\end{equation*}
and $\Var\big(\sum_{i=1}^{l}a_i\xi_i\big)=\Var\big(\sum_{j=1}^lb_j\xi_{k+l+j}\big)=1$. The triangular inequality yields
\begin{multline*}
\bigg|\Cov\Big(\sum\limits_{i=1}^la_i \xi_i, \sum_{j=1}^lb_j\xi_{k+l+j}\Big)\bigg|\\
\leq \Big|\sum\limits_{i=1}^la_i\sum\limits_{j=1}^{l}b_j\Big|\left|\gamma(k)\right|
+\sum\limits_{i=1}^l\sum_{j=1}^{l}|a_i||b_j|\left|\gamma(k)-\gamma(k+l+j-i)\right|.
\end{multline*}
We will treat the two summands on the right hand side separately. For the first term, it follows by Lemma \ref{lemB} that 
\begin{equation*}
\Big|\sum\limits_{i=1}^la_i\sum_{j=1}^{l}b_j\Big|\left|\gamma(k)\right|=\Big|\sum\limits_{i=1}^la_i\Big|\Big|\sum_{j=1}^{l}b_j\Big|\left|\gamma(k)\right|\leq K'^2_dl^{D}L_\gamma(k)k^{-D}.
\end{equation*}
Before we deal with the second summand, we  observe that by H\"older's inequality and Lemma \ref{lemA}
\begin{equation*}
\sum\limits_{i=1}^l|a_i|\leq \sqrt{l\sum\limits_{i=1}^la_i^2}\leq \sqrt{K_D}\sqrt{l}\ \ \text{and} \ \ \sum\limits_{j=1}^{l}|b_j|\leq \sqrt{l\sum\limits_{j=1}^lb_j^2}\leq\sqrt{K_D}\sqrt{l}.
\end{equation*}
Due to  Assumption \ref{dgp}
\begin{equation*}
\sup\limits_{|k-\tilde{k}|\leq 2l-1}\left|L_{\gamma}(k)-L_{\gamma}(\tilde{k})\right|\leq K\frac{l}{k}
\end{equation*}
for some constant $K$.

Consequently, for all $\tilde{k} \in \left\{k+1, \ldots, k+2l-1\right\}$
\begin{align*}
\left|\gamma(k)-\gamma(\tilde{k})\right|
&\leq  L_\gamma(k) \left|k^{-D}-\tilde{k}^{-D}\right|+|L_\gamma(k)-L_\gamma(\tilde{k})|\tilde{k}^{-D}\\
&\leq  L_\gamma(k) \left(k^{-D}-(k+2l-1)^{-D}\right)+|L_\gamma(k)-L_\gamma(\tilde{k})|k^{-D}\\
&\leq C_d k^{-D-1}lL_\gamma(k)+K\frac{l}{k}k^{-D}\max\{L_{\gamma}(k),1\}\\
&\leq C_3k^{-D-1}l\max\{L_{\gamma}(k),1\}
\end{align*}
for some constants $C_d$, $C_3$. Combining this with the bounds for $\sum_{i=1}^l|a_i|$, $\sum_{j=1}^{l}|b_j|$, we finally arrive at
\begin{align*}
\sum\limits_{i=1}^l|a_i|\sum\limits_{j=1}^{l}|b_j|\left|\gamma(k)-\gamma(k+l+j-i)\right|
&\leq K_Dl\max\limits_{\tilde{k} \in \left\{k+1, \ldots, k+2l-1\right\}}\left|\gamma(k)-\gamma(\tilde{k})\right|\\
&=K_DC_3k^{-D-1}l^2\max\{L_{\gamma}(k),1\}.
\end{align*}
\end{proof}

\subsection{Proof of the Main Result}

Let $t$ be a point of continuity of $F_T$. In order to simplify notation, we write $N=n-l+1$ and $T_{l, i}=T_l(X_i, \ldots, X_{i+l-1})$. The triangular inequality yields
\begin{align*}
|\hat{F}_{l, n}(t)-F_{T_n}(t)|\leq |\hat{F}_{l, n}(t)-F_T(t)|+|F_{T}(t)-F_{T_n}(t)|.
\end{align*}
The second term on the right-hand side of the above inequality converges to zero because of Assumption \ref{convergence of T_n}. As $L_2$-convergence implies stochastic convergence,
it suffices to show that 
 \begin{align*}
 \E\left(|\hat{F}_{l, n}(t)-F_T(t)|^2\right)\longrightarrow 0
 \end{align*}
in order to prove that the first term converges to zero, as well.
We have
 \begin{align*}
  &\E\left(|\hat{F}_{l, n}(t)-F_T(t)|^2\right)\\
    &=\E \left(\hat{F}_{l, n}^2(t)\right)-\left(\E \hat{F}_{l, n}(t)\right)^2+\left(F_T(t)\right)^2-2F_T(t)\E \hat{F}_{l, n}(t)+\left(\E \hat{F}_{l, n}(t)\right)^2\\
    &=\Var(\hat{F}_{l, n}(t))+\left|\E \hat{F}_{l, n}(t)-F_T(t)\right|^2.
 \end{align*}
Furthermore, stationarity of the process $\left(X_n\right)_{n\in \mathbb{N}}$ and  Assumption \ref{convergence of T_n} imply
 \begin{equation*}
\E \hat{F}_{l, n}(t) =\frac{1}{N}\sum\limits_{i=1}^N \E\left(1_{\left\{T_{l, i}\leq t\right\}}\right)
  =P\left(T_{l, 1}\leq t\right)  =F_{T_l}(t)\xrightarrow{l\rightarrow\infty} F_T(t).
 \end{equation*}
It remains  to show that   $\Var (\hat{F}_{l, n}(t))\longrightarrow 0$.
Again,  it follows by stationarity of $\left(X_n\right)_{n\in \mathbb{N}}$ that
 \begin{align*}
 \Var \left(\hat{F}_{l, n}(t)\right)
              &=\frac{1}{N}\Var\left(1_{\left\{T_{l, 1}\leq t\right\}}\right)+\frac{2}{N^2}\sum\limits_{i=2}^{N}(N-i+1)\Cov\left(1_{\left\{T_{l, 1}\leq t\right\}}, 1_{\left\{T_{l, i}\leq t\right\}}\right)\\
 &\leq  \frac{2}{N}\sum\limits_{i=1}^{N}\left|\Cov\left(1_{\left\{T_{l, 1}\leq t\right\}}, 1_{\left\{T_{l, i}\leq t\right\}}\right)\right|.
 \end{align*}
Recall that by Assumption \ref{blocklength}, we have $l\leq C_ln^{(1+D)/2-\epsilon}$ for some constants $C_l$ and $\epsilon>0$. For $n$ large enough such that $l<\frac{1}{2}\lfloor n^{1-\epsilon/2}\rfloor$, we split the sum of covariances into two parts:
\begin{align*}
&\frac{1}{N}\sum\limits_{i=1}^{N}\left|\Cov\left(1_{\left\{T_{l, 1}\leq t\right\}}, 1_{\left\{T_{l, i}\leq t\right\}}\right)\right|\\
&=\frac{1}{N}\sum\limits_{i=1}^{\lfloor n^{1-\epsilon/2}\rfloor}\!\!\left|\Cov\left(1_{\left\{T_{l, 1}\leq t\right\}}, 1_{\left\{T_{l, i}\leq t\right\}}\right)\right|
+\frac{1}{N}\sum\limits_{i=\lfloor n^{1-\epsilon/2}\rfloor+1}^{N}\!\!\!\!\left|\Cov\left(1_{\left\{T_{l, 1}\leq t\right\}}, 1_{\left\{T_{l, i}\leq t\right\}}\right)\right|\\
&\leq\frac{\lfloor n^{1-\epsilon/2}\rfloor}{N}+\frac{1}{N}\sum\limits_{k=\lfloor n^{1-\epsilon/2}\rfloor+1}^{N}\rho(\sigma(X_i, 1\leq i\leq l), \sigma(X_j, k\leq j\leq k+l-1))\\
&\leq \frac{\lfloor n^{1-\epsilon/2}\rfloor}{N}+\frac{1}{N}\sum\limits_{k=\lfloor n^{1-\epsilon/2}\rfloor-l}^{N-l-1}\rho(k, l),
\end{align*} 
where
\begin{align*}
\rho(k,l):=\rho\big(\sigma(X_i, 1\leq i\leq l), \sigma(X_j, k+l+1\leq j\leq k+2l)\big).
\end{align*}
Obviously, the first summand converges to zero by Assumption \ref{blocklength}. For the second summand note that as a consequence of Potter's Theorem (Theorem 1.5.6 in the book of Bingham, Goldie and Teugels \citep{Bingham1987}), there is a constant $C_L$ such that $L_{\gamma}(k)\leq C_Lk^{D\epsilon/2}$ for all $k\in \mathbb{N}$. This together with Lemma \ref{LemC} yields
\begin{align*}
&\frac{1}{N}\sum\limits_{k=\lfloor n^{1-\epsilon/2}\rfloor-l}^{N-l-1}\rho(k, l)\\
&\leq C_LC_1\frac{l^D}{N}\sum\limits_{k=\lfloor n^{1-\epsilon/2}\rfloor/2}^{N-l-1}k^{-D}k^{D\epsilon/2}+C_LC_2\frac{l^2}{N}\sum\limits_{k=\lfloor n^{1-\epsilon/2}\rfloor/2}^{N-l-1}k^{-D-1}k^{D\epsilon/2}\\
&\leq C_LC_1C_l^{D}2^{D(1-\epsilon/2)}n^{D\left(((1+D)/2-\epsilon)-(1-\epsilon/2)+\epsilon/2(1-\epsilon/2)\right)}\\
& \ \quad +C_LC_2C_l^22^{1+D(1-\epsilon/2)}n^{\left((1+D-2\epsilon)-(D+1)(1-\epsilon/2)+(1-\epsilon/2)D\epsilon/2\right)}\\
&\leq C\left(n^{-D\left((1-D)/2+\epsilon^2/4\right)}+n^{-\epsilon(\frac{3}{2}-D+D\epsilon/4)}\right)\xrightarrow{n\rightarrow\infty}0
\end{align*}
for some constant $C<\infty$. Thus, we have proved that $\Var(\hat{F}_{l, n}(t))\rightarrow 0$ as $n\rightarrow \infty$ and that the first conjecture of Theorem \ref{main result} holds.

The second assertion of Theorem \ref{main result} follows 
from
\begin{align*}
F_{T_n}(t)-\hat{F}_{l,n}(t)\overset{\mathcal{P}}{\longrightarrow}0
\end{align*}
by the usual Glivenko-Cantelli argument for the uniform convergence of empirical distribution functions; see for example section 20 in the book of Billingsley \cite{Billingsley1995}. \qed

\section*{Acknowledgements}

We thank the referee for his careful reading of the article and his thoughtful comments which lead to a significant improvement of the article. We also thank Norman Lambot for reading the article, thereby helping to reduce the number of misprints.

\bibliographystyle{amsplain}
\bibliography{PaperAM}
\pagebreak

\appendix
\section{A Modified Change Point Test for Data with Ties}\label{appA}

If the distribution of $X_i=G(\xi_i)$ is not continuous, there is a positive probability that $X_i=X_j$ for some $i\neq j$, so there might be ties in the sample. We propose to use the following test statistic based on the modified ranks $\tilde{R}_i=\sum_{j=1}^n(1_{\{X_j< X_i\}}+\frac{1}{2}1_{\{X_j= X_i\}})$:
\begin{equation*}
\tilde{T}_n(\tau_1, \tau_2):=\max_{k\in \left\{\lfloor n\tau_1\rfloor, \ldots,  \lfloor n\tau_2\rfloor\right\}}\frac{\left|\sum_{i=1}^k\tilde{R}_i-\frac{k}{n}\sum_{i=1}^n\tilde{R}_i\right|}{\Big\{\frac{1}{n}\sum_{t=1}^k \tilde{S}_t^2(1,k)+\frac{1}{n}\sum_{t=k+1}^n \tilde{S}_t^2(k+1,n)\Big\}^{1/2}}, 
\end{equation*}
where
\begin{equation*}
\tilde{S}_{t}(j, k)=\sum\limits_{h=j}^t\bigg(\tilde{R}_h-\frac{1}{k-j+1}\sum\limits_{i=j}^k\tilde{R}_i\bigg).
\end{equation*}
To be able to apply subsampling, we need $\tilde{T}_n$ to converge in distribution, which we will show now:

\begin{Lem}  Let $(\xi_n)_{n\in\N}$ be a stationary sequence of centered standard Gaussian variables with covariance function $\gamma(k)=k^{-D}L_\gamma(k)$ for a $D\in(0,1)$ and a slowly varying function $L_{\gamma}$. Let $X_i=G(\xi_i)$ for a function $G$,  piecewise monotone on finitely many pieces. Then $\tilde{T}_n(\tau_1, \tau_2)\Rightarrow T$ for some random variable $T$.
\end{Lem}

\begin{proof} Let $h(x,y)=1_{\{G(x)< G(y)\}}+\frac{1}{2}1_{\{G(x)=G(y)\}}-\frac{1}{2}$. We define the modified Wilcoxon process $(\tilde{W}_n(\lambda))_{\lambda\in[0,1]}$ by
\begin{equation*}
\tilde{W}_n(\lambda):=\frac{1}{nd_n}\sum_{i=1}^{[n\lambda]}\sum_{j=[n\lambda]+1}^nh(\xi_i,\xi_j)
\end{equation*}
with $d_n=\sqrt{\Var(\sum_{i=1}^n\xi_i)}$. From Theorem 2.2 in Dehling, Rooch, Wendler \citep{dehling2014two}, we have the weak convergence of this process $\tilde{W}_n$ to the limit process $W$ with
\begin{multline*}
W(\lambda)\\
=-(1-\lambda)Z(\lambda)\!\int\!\varphi(x)d\tilde{h}(x)-\lambda(Z(1)-Z(\lambda))\!\int\!\Big(\int\! \varphi(y)dh(x,y)(y)\Big)\varphi(x)dx.
\end{multline*}
Here, $Z$ is a fractional Brownian motion, $\varphi$ is the density function of the standard normal distribution and $\tilde{h}(x)=E[h(x,\xi_i)]$. Following the proof of Theorem 1 in  \citep{Betken2014}, we can express $\tilde{T}_n(\tau_1, \tau_2)$ as a function of $\tilde{W}_n$:
\begin{multline*}
T_n(\tau_1, \tau_2)\\
 =\sup_{\tau_1\leq\lambda\leq \tau_2}\frac{\left|\tilde{W}_n(\lambda)\right|}{\big\{\int_0^\lambda (\tilde{W}_n(t)-\frac{c_n(t)}{c_n(\lambda)}\tilde{W}_n(\lambda))^2dt+\int_\lambda^{1} (\tilde{W}_n(t)-\frac{1-c_n(t)}{1-c_n(\lambda)}\tilde{W}_n(\lambda))^2dt\big\}^{1/2}}.
\end{multline*}
Note that $c_n(\lambda)$ converges to $\lambda$ uniformly, so we have  the asymptotic equivalence
\begin{multline*}
T_n(\tau_1, \tau_2)\\
 \approx\sup_{\tau_1\leq\lambda\leq \tau_2}\frac{\left|\tilde{W}_n(\lambda)\right|}{\big\{\int_0^\lambda (\tilde{W}_n(t)-\frac{t}{\lambda}\tilde{W}_n(\lambda))^2dt+\int_\lambda^{1} (\tilde{W}_n(t)-\frac{1-t}{1-\lambda}\tilde{W}_n(\lambda))^2dt\big\}^{1/2}}.
\end{multline*}
By the continuous mapping theorem, we get 
\begin{multline*}
T_n(\tau_1, \tau_2)\\
 \Rightarrow\sup_{\tau_1\leq\lambda\leq \tau_2}\frac{\left|W(\lambda)\right|}{\big\{\int_0^\lambda (W(t)-\frac{t}{\lambda}W(\lambda))^2dt+\int_\lambda^{1} (W(t)-\frac{1-t}{1-\lambda}W(\lambda))^2dt\big\}^{1/2}}=:T.
\end{multline*}
\end{proof}

\section{A Test for Multiple Change Points}\label{appB}

For testing the alternative hypothesis of two change points, we suggest to use the test statistic $T_n(\tau_1, \tau_2, \varepsilon)=\sup_{(k_1, k_2)\in \Omega_n(\tau_1, \tau_2, \varepsilon)}
\left|G_n(k_1, k_2)\right|$. Some calculations yield
\begin{align*}
&G_n(k_1, k_2)\\
\ignore{=&\frac{\left|\sum_{i=1}^{k_1}R_{i}^{(1)}-\frac{k_1}{k_2}\sum_{i=1}^{k_2}R_{i}^{(1)}\right|}{\bigg\{\frac{1}{n}\sum_{t=1}^{k_1} \left(S_{t}^{(1)}(1,k_1)\right)^2+\frac{1}{n}\sum_{t=k_1+1}^{k_2} \left(S_{t}^{(1)}(k_1+1,k_2)\right)^{(1)}\bigg\}^{1/2}}\\
&+\frac{\left|\sum_{i=k_1+1}^{k_2}R_{i}^{(2)}-\frac{k_2-k_1}{n-k_1}\sum_{i=k_1+1}^{n}R_{i}^{(2)}\right|}{\bigg\{\frac{1}{n}\sum_{t=k_1+1}^{k_2} \left(S_{t}^{(2)}(k_1+1,k_2)\right)^{2}+\frac{1}{n}\sum_{t=k_2+1}^{n} \left(S_{t}^{(2)}(k_1+1,n)\right)^{2}\bigg\}^{1/2}}\\}
\ignore{&=\frac{\left|W_n(\lambda_1, \lambda_1)-W_n(\lambda_1, \lambda_2)\right|}{\bigg\{\int_0^{\lambda_1} \left(W_n(r, r)-W_n(r, \lambda_2)-\frac{r}{\lambda_1}\left\{W_n(\lambda_1, \lambda_1)-W_n(\lambda_1, \lambda_2)\right\}\right)^2dr+ \int_{\lambda_1}^{\lambda_2}\left(W_n(r, r)-W_n(r, \lambda_2)-\frac{\lambda_2-r}{\lambda_2-\lambda_1}\left(W_n(\lambda_1, \lambda_1)-W_n(\lambda_1, \lambda_2)\right)\right)^2dr\bigg\}^{\frac{1}{2}}}\\
&+\frac{\left|W_n(\lambda_2, \lambda_2)-W_n(\lambda_1, \lambda_2)\right|}{\bigg\{\int_{\lambda_1}^{\lambda_2}\left(W_n(r, r)-W_n(\lambda_1, r)-\frac{r-\lambda_1}{\lambda_2-\lambda_1}\left(W_n(\lambda_2, \lambda_2)-W_n(\lambda_1, \lambda_2)\right)\right)^2dr+\int_{\lambda_2}^1\left(W_n(r, r)-W_n(\lambda_1, r)-\frac{1-r}{1-\lambda_2}\left\{W_n(\lambda_2, \lambda_2)-W_n(\lambda_1, \lambda_2)\right\}\right)^2dr\bigg\}^{\frac{1}{2}}}\\
&+o_P(1),}
=&\frac{\left|\tilde{W}_n(\lambda_1, \lambda_2)\right|}{\bigg\{\int\limits_0^{\lambda_1}\!\big(\tilde{W}_n(r, \lambda_2)\!-\!\frac{r}{\lambda_1}\tilde{W}_n(\lambda_1, \lambda_2)\big)^2dr+ \int\limits_{\lambda_1}^{\lambda_2}\!\big(\tilde{W}_n(r, \lambda_2)\!-\!\frac{\lambda_2-r}{\lambda_2-\lambda_1}\tilde{W}_n(\lambda_1, \lambda_2)\big)^2dr\bigg\}^{\frac{1}{2}}}\\
&+\frac{\left|W^*_n(\lambda_2, \lambda_1)\right|}{\bigg\{\int\limits_{\lambda_1}^{\lambda_2}\!\big(W^*_n(r, \lambda_1)\!-\!\frac{r\!-\!\lambda_1}{\lambda_2\!-\!\lambda_1}W^*_n(\lambda_2, \lambda_1)\big)^2dr+\int\limits_{\lambda_2}^1\!\big(W^*_n(r, \lambda_1)\!-\!\frac{1\!-\!r}{1\!-\!\lambda_2}W^*_n(\lambda_2, \lambda_1)\big)^2dr\bigg\}^{\frac{1}{2}}}\\
&+o_P(1),
\end{align*}
where
\begin{align*}
&\tilde{W}_n(\lambda, \tau):=W_n(\lambda, \lambda)-W_n(\lambda, \tau), \quad
W^*_n(\lambda, \tau):=W_n(\lambda, \lambda)-W_n(\tau, \lambda)
\end{align*}
with
\begin{align*}
W_n(\lambda, \tau)=\sum\limits_{i=1}^{\lfloor n\lambda\rfloor}\sum\limits_{j=\lfloor n\tau\rfloor +1}^n\left(1_{\left\{X_i\leq X_j\right\}}-\frac{1}{2}\right), \ 0\leq \lambda\leq\tau\leq 1.
\end{align*}
Define
\begin{align*}
d_n^2:=\Var\left(\sum\limits_{j=1}^nH_r(\xi_j)\right),
\end{align*}
where $H_r$ denotes the $r$-th order Hermite polynomial and $r$ designates the Hermite rank of the class of functions $\left\{1_{\left\{G(\xi_i)\leq x\right\}}-F(x), \ x\in \mathbb{R}\right\}$.
It can be shown that $\frac{1}{nd_n}W_n(\lambda, \tau)$ converges in distribution to
\begin{align*}
\left\{(1-\tau)Z_r(\lambda)-\lambda(Z_r(1)-Z_r(\tau))\right\}\frac{1}{r!}\int J_r(x)dF(x), \ 0\leq \lambda\leq\tau\leq 1,
\end{align*}
where $Z_r$ is an $r$-th order Hermite process with Hurst parameter $H:=\max\{1-\frac{rD}{2},\frac{1}{2}\}$
and where
\begin{align*}
J_r(x)=\E \left(H_r(\xi_i)1_{\left\{G(\xi_i)\leq x\right\}}\right).
\end{align*}
As a result, under the hypothesis the limiting distribution of $T_n(\tau_1, \tau_2, \varepsilon)$ is given by $T(r, \tau_1, \tau_2, \varepsilon)=\sup_{\tau_1\leq \lambda_1<\lambda_2\leq \tau_2, \ \lambda_2-\lambda_1\geq\varepsilon}G_{r}(\lambda_1, \lambda_2)$ with
\begin{align*}
&G_{r}(\lambda_1, \lambda_2)\\
\ignore{&=\frac{\left|\lambda_2 Z_r(\lambda_1)-\lambda_1 Z_r(\lambda_2)\right|}{\bigg\{\int\limits_{0}^{\lambda_1}\!\big(\lambda_2 \big(Z_r(t)-\frac{t}{\lambda_1}Z_r(\lambda_1)\big)\big)^2dt+\int\limits_{\lambda_1}^{\lambda_2}\!\big(\lambda_2 \big(Z_r(t)-\frac{t-\lambda_1}{\lambda_2-\lambda_1}Z_r(\lambda_2)-\frac{\lambda_2-t}{\lambda_2-\lambda_1}Z_r(\lambda_1)\big)\big)^2dt\bigg\}^{\frac{1}{2}}}\\
&+\frac{\left|Z_r(\lambda_2)-Z_r(\lambda_1)+\lambda_1\left(Z_r(1)-Z_r(\lambda_2)\right)-\lambda_2\left(Z_r(1)-Z_r(\lambda_1)\right)\right|}{\bigg\{\int\limits_{\lambda_1}^{\lambda_2}\!\big((1\!-\!\lambda_1)\big(Z_r(t)\!-\!\frac{\lambda_2\!-\!t}{\lambda_2\!-\!\lambda_1}Z_r(\lambda_1)\!-\!\frac{t\!-\!\lambda_1}{\lambda_2\!-\!\lambda_1}Z_r(\lambda_2)\big)\big)^2dt+\int\limits_{\lambda_2}^{1}\!\big((1\!-\!\lambda_1)\big(Z_r(t)\!-\!\frac{1\!\!-t}{1\!-\!\lambda_2}Z_r(\lambda_2)\!-\!\frac{t\!-\!\lambda_2}{1\!-\!\lambda_2}Z_r(1)\big)\big)^2dt\bigg\}^{\frac{1}{2}}}\\}
&=\frac{\left| Z_r(\lambda_1)-\frac{\lambda_1}{\lambda_2} Z_r(\lambda_2)\right|}{\bigg\{\int\limits\!_{0}^{\lambda_1}\big(Z_r(t)-\frac{t}{\lambda_1}Z_r(\lambda_1)\big)^2dt+\int\limits\!_{\lambda_1}^{\lambda_2}\big(Z_r(t)-\frac{t-\lambda_1}{\lambda_2-\lambda_1}Z_r(\lambda_2)-\frac{\lambda_2-t}{\lambda_2-\lambda_1}Z_r(\lambda_1)\big)^2dt\bigg\}^{\frac{1}{2}}}\\
&\quad +\frac{\left|Z_r(\lambda_2)-\frac{1-\lambda_2}{1-\lambda_1}Z_r(\lambda_1)-\frac{\lambda_2-\lambda_1}{1-\lambda_1}Z_r(1)\right|}{\bigg\{\int\limits_{\lambda_1}^{\lambda_2}\!\big(Z_r(t)-\frac{\lambda_2-t}{\lambda_2-\lambda_1}Z_r(\lambda_1)-\frac{t-\lambda_1}{\lambda_2-\lambda_1}Z_r(\lambda_2)\big)^2dt+\int\limits_{\lambda_2}^{1}\!\big(Z_r(t)-\frac{1-t}{1-\lambda_2}Z_r(\lambda_2)-\frac{t-\lambda_2}{1-\lambda_2}Z_r(1)\big)^2dt\bigg\}^{\frac{1}{2}}}.
\end{align*}

\end{document}